\documentclass[12pt,reqno]{amsart}
\usepackage{cite}
\usepackage{amsfonts}
\usepackage{amssymb, amsmath, latexsym}
\usepackage{lipsum}
\usepackage{hyperref}
\usepackage{mleftright}

\usepackage{mathtools}

\usepackage{array}

\newcommand\cplus{\mathbin{\raisebox{-\height}{$+$}}}
\newcommand\contdots{\raisebox{-\height}{$\vphantom{+}\dotsm$}}

\numberwithin{equation}{section}
\theoremstyle{plain}
\newtheorem{theorem}{Theorem}[section]
\newtheorem{lemma}[theorem]{Lemma}
\newtheorem{conjecture}[theorem]{Conjecture}

\newtheorem{remark}[theorem]{Remark}
\newcommand\T{\rule{0pt}{2.6ex}}
\newcommand\B{\rule[-1.2ex]{0pt}{0pt}}

\usepackage{pgf}
\usepackage[absolute, overlay]{textpos}

\textblockorigin{\paperwidth}{0.0 pt}

\title[Infinite $q$-products with matching coefficients] {Matching coefficients in the series expansions of certain $q$-products and their reciprocals}
\author{Nayandeep Deka Baruah}
\address{Department of Mathematical Sciences, Tezpur University, Napaam 784028, Sonitpur, Assam, India}
\email{nayan@tezu.ernet.in}
\author{Hirakjyoti Das}
\address{Department of Mathematical Sciences, Tezpur University, Napaam 784028, Sonitpur, Assam, India}
\email{hdas@tezu.ernet.in}

\allowdisplaybreaks
\begin{document}

\begin{abstract}
We show that the series expansions of certain  $q$-products have \textit{matching coefficients} with their reciprocals. Several of the results are associated to Ramanujan's continued fractions. For example, let $R(q)$ denote the Rogers-Ramanujan continued fraction having the well-known $q$-product repesentation $$R(q)=\dfrac{(q;q^5)_\infty(q^4;q^5)_\infty}{(q^2;q^5)_\infty(q^3;q^5)_\infty}.$$ If
\begin{align*}
\sum_{n=0}^{\infty}\alpha(n)q^n=\dfrac{1}{R^5\left(q\right)}=\left(\sum_{n=0}^{\infty}\alpha^{\prime}(n)q^n\right)^{-1},\\
\sum_{n=0}^{\infty}\beta(n)q^n=\dfrac{R(q)}{R\left(q^{16}\right)}=\left(\sum_{n=0}^{\infty}\beta^{\prime}(n)q^n\right)^{-1},
\end{align*}
then
\begin{align*}
\alpha(5n+r)&=-\alpha^{\prime}(5n+r-2) \quad r\in\{3,4\},\\
\beta(10n+r)&=-\beta^{\prime}(10n+r-6) \quad r\in\{7,9\}.
\end{align*}
\end{abstract}
\maketitle

\vskip 3mm
\noindent {\footnotesize 2010 Mathematical Reviews Classification Numbers: 11F33, 11A55}

\section{Introduction and results}\label{Introduction}
For  complex numbers $a$ and $q$  such that $\mid q\mid<1$, we customarily define
\begin{align*}\left(a;q\right)_n:=\prod_{j=0}^{n-1}\left(1-a q^j\right), \quad \left(a;q\right)_{\infty}:=\prod_{j=0}^{\infty}\left(1-a q^j\right).
\end{align*}
For brevity, we set $\left(a_1,a_2,\ldots,a_n;q\right)_{\infty}:=\left(a_1;q\right)_{\infty}\left(a_2;q\right)_{\infty}\cdots\left(a_n;q\right)_{\infty}$.

For two  power series $\sum_{n=0}^\infty A(n)q^n$ and $\sum_{n=0}^\infty B(n)q^n$, if for some positive integers $a,c,$ and $k$  and non-negative integers $b$ and $d$,  $A(an+b)=\pm k B(cn+d)$,  for all $n\ge0,$  then the two power series are said to have matching coefficients.

We  notice that the series expansions of certain  $q$-products have matching coefficients with their reciprocals. For example, let
$$\sum_{n=0}^{\infty}\alpha(n)q^n= \left(q;q^2\right)_{\infty}^8= \left(\sum_{n=0}^{\infty}\alpha^{\prime}(n)q^n\right)^{-1}.$$ Equating the coefficients of $q^{2n+1}$ in the expansions of the  $q$-products of Jacobi's famous identity \cite[p. 470]{ww}
\begin{align}\label{jacobi-quartic}
\left(q;q^2\right)_{\infty}^8-\left(-q;q^2\right)_{\infty}^8=-16q\left(q^2;q^4\right)_{\infty}^{-8},
\end{align}
one can arrive at
\begin{align*}
\alpha(2n+1)=-8\alpha^\prime(n).
\end{align*}

Again, cubing \eqref{jacobi-quartic}, we find that
\begin{align*}
\left(q;q^2\right)_{\infty}^{24}-\left(-q;q^2\right)_{\infty}^{24}=-4096q^3\left(q^2;q^4\right)_{\infty}^{-24}-48q,
\end{align*}
from which, it follows that
\begin{align*}
\beta(2n+3)=-2048\beta^\prime(n),
\end{align*}
where
$$\sum_{n=0}^{\infty}\beta(n)q^n= \left(q;q^2\right)_{\infty}^{24}= \left(\sum_{n=0}^{\infty}\beta^{\prime}(n)q^n\right)^{-1}.$$

 In this paper, we find many such $q$-products having matching coefficients with their reciprocals. The results mainly arise from Ramanujan's theta function identities, modular equations, and identities associated to Ramanujan's continued fractions. In the process, we also find some new identities for the Rogers-Ramanujan type functions.

In the following subsections we present our results. For convenience,  if $A(n)$  is the $n$-th coefficient of the series expansion of a $q$-product, then  the corresponding $n$-th coefficient for the  reciprocal is denoted by $A^\prime(n)$.

\subsection{Matching coefficients arising from Ramanujan's theta functions}
Ramanujan's general theta function $f(a,b)$ is defined by
\begin{align}\label{jtpi}
f(a,b):=&\sum_{n=-\infty}^{{\infty}}a^{n(n+1)/2}b^{n(n-1)/2}=(-a,-b,ab;ab)_{\infty}, \quad \mid ab \mid<1,
\end{align}
where the last equality is Jacobi's famous triple product identity. Three special cases of $f(a,b)$ are
\begin{align*}
f(-q)&:=f(-q,-q^{2})=f_1, \quad
\varphi(q):=f(q,q)=\dfrac{f_2^5}{f_1^2f_4^2}, \quad
\psi(q):=f\left(q,q^3\right)=\dfrac{f_2^2}{f_1},
\end{align*}
where for convenience, we set $f_n:=\left(q^n;q^n\right)_\infty$.

The following results arise from the $p$-dissections of $f(-q)$, $\varphi(q)$, and $\psi(q)$, where for a power series  $P(q)$ in $q$,  a \textit{$p$-dissection} of $P(q)$ is given by
\begin{align*}
P(q)= \sum_{j=0}^{p-1}q^{j}P_j\left(q^p\right),
\end{align*}
where $P_j$'s are power series in $q^p$.

\begin{theorem}\label{Theo_1}
For odd primes $p$, suppose that
\begin{align*}
&\sum_{n=0}^{\infty}\gamma_{1}(n)q^n=\left(q,q^2,\ldots,q^{p-1};q^{p}\right)_{\infty},\quad\sum_{n=0}^{\infty}\gamma_2(n)q^n=\dfrac{\left(q,q^2,\ldots,q^{p-1};q^{p}\right)_{\infty}^{2}}{\left(q^2,q^4,\ldots,q^{2(p-1)};q^{2p}\right)_{\infty}},\\
&\sum_{n=0}^{\infty}\gamma_3(n)q^n=\dfrac{\left(q^2,q^4,\ldots,q^{2(p-1)};q^{2p}\right)_{\infty}^{2}}{\left(q,q^2,\ldots,q^{p-1};q^{p}\right)_{\infty}}.
\end{align*}
 Then, for any $n\geq 0$, we have
\begin{align*}
\textup{(i)}& ~\gamma_{1}\left(pn+\dfrac{p^2-1}{24}\right)=(-1)^{\frac{p\pm1}{6}}\gamma_{1}^{\prime}(n), \quad where~ p>3;
\quad\textup{(ii)} ~\gamma_2(pn)=\gamma^{\prime}_{2}(n);\\
 \quad\textup{(iii)}& ~\gamma_3\left(pn+\dfrac{p^2-1}{8}\right)=\gamma^{\prime}_{3}(n),
\end{align*}
where the  sign $\pm$ in (i) should be chosen according to $(p\pm1)/6$ being an integer.
\end{theorem}

The 3-, 5-, 7-, and 11-dissections of $f_1^k$, for $k\geq1$ give further $q$-products having matching coefficients with their reciprocals. We state some of the results in the following theorem.

\begin{theorem}\label{ExtraTheo1}
For integers  $\ell>1, m>1$, and $k\ge1$, suppose that
\begin{align*}
\sum_{n=0}^{\infty}\delta_{m;k}(n)q^n&=\dfrac{f_1^k}{f_m^k},&
\sum_{n=0}^{\infty}\delta_{\ell,m;k}(n)q^n&=\dfrac{f_1^k f_{\ell}^k}{f_m^k f_{\ell m}^k}.
\end{align*}
Then, for all $n\geq 0$, we have
\begin{align*}
&\textup{(i)} & &\delta_{3;3}\left(3n+1\right)=-3\delta_{3;3}^{\prime}(n);&
&\textup{(ii)} & &\delta_{3;6}\left(3n+2\right)=9\delta_{3;6}^{\prime}(n);\\
&\textup{(iii)} & &\delta_{5;2}\left(5n+2\right)=-\delta_{5;2}^{\prime}(n); &
&\textup{(iv)} & &\delta_{5;3}\left(5n+3\right)=5\delta_{5;3}^{\prime}(n);\\
&\textup{(v)} & &\delta_{5;4}\left(5n+4\right)=-5\delta_{5;4}^{\prime}(n);&
&\textup{(vi)} & &\delta_{5;8}\left(5n+3\right)=-125\delta_{5;8}^{\prime}(n);\\
&\textup{(vii)} & &\delta_{7;2}\left(7n+4\right)=\delta_{7;2}^{\prime}(n);&
&\textup{(viii)} & &\delta_{7;3}\left(7n+6\right)=-7\delta_{7;3}^{\prime}(n);\\
&\textup{(ix)} & &\delta_{7;6}\left(7n+12\right)=49\delta_{7;6}^{\prime}(n); &
&\textup{(x)} & &\delta_{11;2}\left(11n+10\right)=\delta_{11;2}^{\prime}(n);\\
&\textup{(xi)} & &\delta_{2,3;1}\left(3n+1\right)=-\delta_{2,3;1}^{\prime}(n); &
&\textup{(xii)} & &\delta_{2,3;2}\left(3n+2\right)=-3\delta_{2,3;2}^{\prime}(n);\\
&\textup{(xiii)} & &\delta_{2,3;5}\left(3n+5\right)=-81\delta_{2,3;5}^{\prime}(n);&
&\textup{(xiv)} & &\delta_{2,5;1}\left(5n+3\right)=\delta_{2,5;1}^{\prime}(n);\\
&\textup{(xv)} & &\delta_{2,5;3}\left(5n+9\right)=25\delta_{2,5;3}^{\prime}(n);&
&\textup{(xvi)} & &\delta_{3,5;1}\left(5n+4\right)=\delta_{3,5;1}^{\prime}(n);\\
&\textup{(xvi)} & &\delta_{3,5;2}\left(5n+8\right)=-5\delta_{3,5;2}^{\prime}(n);&
&\textup{(xviii)} & &\delta_{3,5;3}\left(5n+12\right)=25\delta_{3,5;3}^{\prime}(n);\\
&\textup{(xix)} & &\delta_{2,7;1}\left(7n+6\right)=\delta_{2,7;1}^{\prime}(n); &
&\textup{(xx)} & &\delta_{4,7;1}\left(7n+10\right)=\delta_{4,7;1}^{\prime}(n);\\
&\textup{(xxi)} & &\delta_{3,11;1}\left(11n+20\right)=\delta_{3,11;1}^{\prime}(n); &
&\textup{(xxii)} & &\delta_{4,11;1}\left(11n+25\right)=\delta_{4,11;1}^{\prime}(n).
\end{align*}
\end{theorem}

\subsection{Matching coefficients arising from  modular equations}

The following results arise from Ramanujan's modular equations of degrees 3, 5, 7, 9, 11, 15, 23, and 25.

\begin{theorem}\label{ExtraTheorem1}
For positive integers  $m$ and $k$, suppose that
\begin{align*}
&\sum_{n=0}^{\infty}\mu_{m,k}(n)q^n=\dfrac{\left(q^k;q^{2k}\right)_{\infty}^{mk}}{\left(q;q^2\right)_{\infty}^m}, \quad\quad\sum_{n=0}^{\infty}\nu_{m,k}(n)q^n=\dfrac{\left(q;q^2\right)_{\infty}^{mk}}{\left(q^k;q^{2k}\right)_{\infty}^m},\\
&\sum_{n=0}^{\infty}\eta_{k}(n)q^n=\dfrac{\left(q;q^2\right)_{\infty}^{3}\left(q^9;q^{18}\right)_{\infty}^{3}}{\left(q^3;q^{6}\right)_{\infty}^{2k}},\quad \sum_{n=0}^{\infty}\zeta_{m,k}(n)q^n=\left(q;q^{2}\right)_{\infty}^k\left(q^m;q^{2m}\right)_{\infty}^k,\\
&\sum_{n=0}^{\infty}\theta_1(n)q^n=\left(q;q^{2}\right)_{\infty}\left(q^3;q^{6}\right)_{\infty}\left(q^5;q^{10}\right)_{\infty}\left(q^{15};q^{30}\right)_{\infty},\\
&\sum_{n=0}^{\infty}\theta_2(n)q^n=\dfrac{\left(q;q^{2}\right)_{\infty}}{\left(q^{25};q^{50}\right)_{\infty}}.
\end{align*}
Then, for any $n\geq0$, we have
\begin{align*}
&\textup{(i)} & &\mu_{1,3}(2n+1)=\mu^\prime_{1,3}(n); & &\textup{(ii)} & &\mu_{3,3}(2n+3)=4\mu^{\prime}_{3,3}(n);\\
&\textup{(iii)} & &\mu_{1,5}(2n+3)=2\mu^{\prime}_{1,5}(n); & &\textup{(iv)} & &\nu_{1,3}(2n)=\nu^{\prime}_{1,3}(n);\\
&\textup{(v)} & & \nu_{1,5}(2n+2)=2\nu^{\prime}_{1,5}(n+1); & &\textup{(vi)} & &\nu_{3,3}(2n)=4\nu^{\prime}_{3,3}(n);\\
&\textup{(vii)} & & \eta_{1}(2n+3)=-2\eta^{\prime}_{}(n); & &\textup{(viii)} & &2\eta_{5}(n)=-\eta^{\prime}_{5}(2n);\\
&\textup{(ix)} & &\zeta_{3,2}(2n+1)=-2\zeta^{\prime}_{3,2}(n); &
&\textup{(x)} & &\zeta_{5,4}(2n+3)=-8\zeta^{\prime}_{5,4}(n);\\
&\textup{(xi)} & &\zeta_{7,1}(2n+1)=-\zeta^{\prime}_{7,1}(n); &
&\textup{(xii)} & &\zeta_{7,3}(2n+3)=-4\zeta^{\prime}_{7,3}(n);\\
&\textup{(xiii)} & &\zeta_{11,2}(2n+3)=-2\zeta^{\prime}_{11,2}(n); &
&\textup{(xiv)} & &\zeta_{23,1}(2n+3)=-\zeta^{\prime}_{23,1}(n);\\
&\textup{(xv)} & &\theta_1(2n+3)=-2\theta_1^{\prime}(n); &
&\textup{(xvi)} & &2\theta_2(10n+3)=-\theta_2^{\prime}(5n+3);\\
&\textup{(xvii)} & &2\theta_2(10n+5)=-\theta_2^{\prime}(5n+4).
\end{align*}
\end{theorem}

Note that
\begin{align*}
\sum_{n=0}^\infty\mu_{1,3}(n)q^n=\dfrac{\left(q^3;q^{6}\right)_{\infty}^{3}}{\left(q;q^2\right)_\infty}=\dfrac{q^{1/3}}{\mathcal{C}(q)},
\end{align*}
where $\mathcal{C}(q)$ is Ramanujan's cubic continued fraction defined by (see \cite[p. 345]{B1})
\begin{align*}
\mathcal{C}(q):=\frac{q^{1/3}}{1} \cplus  \frac{q+q^2}{1}\cplus\frac{q^2+q^4}{1}\cplus  \frac{q^3+q^6}{1}\cplus\contdots.
\end{align*}

\subsection{Matching coefficients arising from  identities associated to Ramanujan's continued fractions and Rogers-Ramanujan type functions}

The famous Rogers-Ramanujan continued fraction $R(q)$ is defined by
\begin{align*}
R(q):=\dfrac{1}{1} \cplus  \frac{q}{1}\cplus\frac{q^2}{1}\cplus  \frac{q^3}{1}\cplus\contdots
\end{align*}
and the Rogers-Ramanujan identities are given by
\begin{align*}
G(q)&:= \sum_{n=0}^{\infty}\dfrac{q^{n^2}}{(q,q)_{n}}=\dfrac{1}{(q,q^4;q^5)_{\infty}}, & H(q)&:= \sum_{n=0}^{\infty}\dfrac{q^{n^2+n}}{(q,q)_{n}}=\dfrac{1}{(q^2,q^3;q^5)_{\infty}},
\end{align*}
where $G(q)$ and $H(q)$ are called the Rogers-Ramanujan functions.
It is well-known that $R(q)$, $G(q)$, and $H(q)$ are connected by
\begin{align}\label{R-GH}
R(q)=\dfrac{H(q)}{G(q)}=\dfrac{(q,q^4,;q^5)_{\infty}}{(q^2,q^3;q^5)_{\infty}}.
\end{align}

The following results on matching coefficients arise from the modular relations involving  $G(q)$ and $H(q)$.

\begin{theorem}\label{Theo5_1} Suppose that
\begin{align*}
\sum_{n=0}^{\infty}\lambda_1(n)q^n&=\dfrac{1}{R^5\left(q\right)},~  \sum_{n=0}^{\infty}\lambda_2(n)q^n=\dfrac{R^5\left(q\right)}{R\left(q^5\right)},~
 \sum_{n=0}^{\infty}\lambda_3(n)q^n=\dfrac{R\left(q^2\right)}{R^2\left(q\right)},\\
\sum_{n=0}^{\infty}\lambda_4(n)q^n&=\dfrac{1}{R\left(q\right)R^2\left(q^2\right)},~ \sum_{n=0}^{\infty}\lambda_5(n)q^n=\dfrac{1}{R\left(q\right)R\left(q^4\right)},\\
\sum_{n=0}^{\infty}\lambda_6(n)q^n&=\dfrac{1}{R\left(-q\right)R\left(-q^4\right)},~\sum_{n=0}^{\infty}\lambda_7(n)q^n=\dfrac{R\left(q^4\right)}{R^4\left(q\right)}, ~ \sum_{n=0}^{\infty}\lambda_8(n)q^n=\dfrac{R^2\left(q\right)}{R^3\left(q^4\right)},\\
\sum_{n=0}^{\infty}\lambda_9(n)q^n&=\dfrac{R\left(q\right)}{R\left(q^{16}\right)},~
\sum_{n=0}^{\infty}\lambda_{10}(n)q^n=\dfrac{R\left(q^4\right)}{R^2\left(q\right)R\left(q^2\right)}.
\end{align*}
Then, for any $n\geq0$, we have
\begin{align*}
&\textup{(i)} &\lambda_{1}(5n+r)&=-\lambda_{1}^{\prime}(5n+r-2), \quad  r\in\{3,4\};\\
&\textup{(ii)} &\lambda_{2}(5n)&=\lambda_{2}^{\prime}(5n);\\
&\textup{(iii)} &  \lambda_{3}(5n+r)&=-\lambda_{3}^{\prime}(5n+r), \quad  r\in\{1,4\};\\
&\textup{(iv)}&\lambda_{4}(5n+r)&=-\lambda_{4}^{\prime}(5n+r-2), \quad r\in\{3,4\};\\
&\textup{(v)} &\lambda_{5}(5n+r)&=\lambda_{5}^{\prime}(5n+r-2), \quad  r\in\{3,4\};\\
&\textup{(vi)} & \lambda_{6}(10n+r)&=\lambda_{6}^{\prime}(10n+r-2), \quad  r\in\{5,7\};\\
&\textup{(vii)}& \lambda_{7}(10n+r)&=-\lambda_{7}^{\prime}(10n+r), \quad  r\in\{1,9\};\\
&\textup{(viii)} & \lambda_{8}(10n+r)&=-\lambda_{8}^{\prime}(10n+r-4), \quad  r\in\{5,9\};\\
&\textup{(ix)}& \lambda_{9}(10n+r)&=\lambda_{9}^{\prime}(10n+r-6), \quad  r\in\{7,9\};\\
&\textup{(x)}& \lambda_{10}(2n+r)&=(-1)^r\lambda_{10}^{\prime}(2n+r),\quad  r\in\{0,1\};\\
&\textup{(xi)}&\lambda_{10}(5n+r)&=\lambda_{10}^{\prime}(5n+r),\quad  r\in\{2,3\}.
\end{align*}
\end{theorem}

Next, the Ramanujan-G\"{o}llnitz-Gordon  continued fraction $\mathcal{G}(q)$ is defined by
\begin{align*}
\mathcal{G}(q):=\dfrac{q^{1/2}}{1+q} \cplus  \frac{q^2}{1+q^3}\cplus\frac{q^4}{1+q^5}\cplus  \frac{q^6}{1+q^7}\cplus\contdots
\end{align*}
and the Ramanujan-G\"{o}llnitz-Gordon identities are given by
\begin{align}\label{Sq}
S(q)&:= \sum_{n=0}^{\infty}\dfrac{(-q;q^2)_nq^{n^2}}{(q^2;q^2)_n}=\dfrac{1}{(q,q^4,q^7;q^8)_{\infty}},\\
\label{Tq} T(q)&:= \sum_{n=0}^{\infty}\dfrac{(-q;q^2)_nq^{n(n+2)}}{(q^2;q^2)_n}=\dfrac{1}{(q^3,q^4,q^5;q^8)_{\infty}},
\end{align}
where $S(q)$ and $T(q)$ are known as the Ramanujan-G\"{o}llnitz-Gordon functions. On page 229 of his second notebook \cite{nb}, Ramanujan  recorded a product representation
of $\mathcal{G}(q)$, namely, $$\mathcal{G}(q)=q^{1/2}\dfrac{T(q)}{S(q)}.$$

The following results on matching coefficients arise from modular relations involving $S(q)$ and $T(q)$.

\begin{theorem}\label{Theo8_1}
Let $k$ be a positive integer. Suppose that
\begin{align*}
\sum_{n=0}^{\infty}\rho_{1,k}(n)q^n&=\dfrac{S^k(q)}{T^k(q)}, &
\sum_{n=0}^{\infty}\rho_{2,k}(n)q^n&=\dfrac{T(q)S\left(q^k\right)}{S(q)T\left(q^k\right)},\\
\sum_{n=0}^{\infty}\rho_{3,k}(n)q^n&=\dfrac{S(q)S\left(q^k\right)}{T(q)T\left(q^k\right)},&
\sum_{n=0}^{\infty}\rho_{4}(n)q^n&=\dfrac{T\left(q^2\right)S\left(q^4\right)}{S\left(q^2\right)T\left(q^4\right)}, \\
\sum_{n=0}^{\infty}\rho_{5,k}(n)q^n&=\dfrac{S^k\left(-q\right)T^k\left(q\right)S^k\left(q^4\right)}{S^k\left(q\right)T^k\left(-q\right)T^k\left(q^4\right)}, &
\sum_{n=0}^{\infty}\rho_{6}(n)q^n&=\dfrac{S\left(-q\right)T\left(q^2\right)S\left(q^3\right)}{T\left(-q\right)S\left(q^2\right)T\left(q^3\right)},\\
\sum_{n=0}^{\infty}\rho_{7}(n)q^n&=\dfrac{T\left(-q\right)T\left(q^2\right)S\left(q^3\right)}{S\left(-q\right)S\left(q^2\right)T\left(q^3\right)}, &
\sum_{n=0}^{\infty}\rho_{8}(n)q^n&=\dfrac{S\left(q\right)S\left(-q^3\right)S\left(q^6\right)}{T\left(q\right)T\left(-q^3\right)T\left(q^6\right)},\\
\sum_{n=0}^{\infty}\rho_{9}(n)q^n&=\dfrac{S\left(q\right)T\left(-q^3\right)S\left(q^6\right)}{T\left(q\right)S\left(-q^3\right)T\left(q^6\right)}.
\end{align*}
Then, for any $n\geq0$, we have
\begin{align*}
&\textup{(i)} & &\rho_{1,1}\left(2n+1\right)=\rho_{1,1}^{\prime}\left(2n\right); &
&\textup{(ii)} & &\rho_{1,1}\left(4n+2\right)=-\rho_{1,1}^{\prime}\left(4n+1\right);\\
&\textup{(iii)} & &\rho_{1,2}\left(2n+3\right)=-\rho_{1,2}^{\prime}\left(2n+1\right); &
&\textup{(iv)} & &\rho_{1,2}\left(8n+4\right)=\rho_{1,2}^{\prime}\left(8n+2\right);\\
&\textup{(v)} & &\rho_{1,2}\left(8n+6\right)=\rho_{1,2}^{\prime}\left(8n+4\right);&
&\textup{(vi)} & &\rho_{1,4}\left(4n+6\right)=\rho_{1,2}^{\prime}\left(4n+2\right);\\
&\textup{(vii)} & &\rho_{2,3}\left(4n+5\right)=-\rho_{2,3}^{\prime}\left(4n+3\right);&
 &\textup{(viii)} & &\rho_{2,5}\left(4n+6\right)=-\rho_{2,5}^{\prime}\left(4n+2\right); \\
&\textup{(ix)} & &\rho_{2,7}\left(4n+7\right)=-\rho_{2,7}^{\prime}\left(4n+1\right);&
&\textup{(x)} & &\rho_{2,9}\left(4n+10\right)=-\rho_{2,9}^{\prime}\left(4n+2\right);\\
&\textup{(xi)} & &\rho_{3,3}\left(4n+6\right)=-\rho_{3,3}^{\prime}\left(4n+2\right);&
&\textup{(xii)} & &\rho_{3,5}\left(4n+7\right)=-\rho_{3,5}^{\prime}\left(4n+1\right);\\
&\textup{(xiii)} & &\rho_{3,9}\left(4n+11\right)=-\rho_{3,9}^{\prime}\left(4n+1\right);&
&\textup{(xiv)} & &\rho_{4}\left(4n+2\right)=-\rho_{4}^{\prime}\left(4n\right);\\
&\textup{(xv)} & &\rho_{5,1}\left(4n+6\right)=-\rho_{5,1}^{\prime}\left(4n+2\right);&
&\textup{(xvi)} & &\rho_{5,1}\left(4n+7\right)=\rho_{5,1}^{\prime}\left(4n+3\right);\\
&\textup{(xvii)} & &\rho_{5,1}\left(8n+4\right)=-\rho_{5,1}^{\prime}\left(8n\right);&
&\textup{(xviii)} & &\rho_{5,1}\left(8n+5\right)=\rho_{5,1}^{\prime}\left(8n+1\right);\\
&\textup{(xix)} & &\rho_{5,2}\left(8n+12\right)=-\rho_{5,2}^{\prime}\left(8n+4\right);&
&\textup{(xx)} & &\rho_{6}\left(4n+5\right)=-\rho_{6}^{\prime}\left(4n+3\right);\\
&\textup{(xxi)} & &\rho_{7}\left(8n\right)=-\rho_{7}^{\prime}\left(8n\right);&
&\textup{(xxii)} & &\rho_{8}\left(8n+13\right)=-\rho_{8}^{\prime}\left(8n+3\right);\\
&\textup{(xxiii)} & &\rho_{9}\left(4n+6\right)=-\rho_{9}^{\prime}\left(4n+2\right).
\end{align*}

\end{theorem}

Our next results are associated to the Rogers-Ramanujan type functions $X(q)$ and $Y(q)$ that appear in  two identities in  Slater's list of Rogers-Ramanujan type identities \cite[Identities (49) and (54)]{slater}. The two identities are
\begin{align}
\label{Xq}X(q)&:= \sum_{n=0}^\infty\dfrac{(-q^2;q^2)_n(1-q^{n+1})q^{n(n+2)}}{(q;q)_{2n+2}}=\dfrac{(q,q^{11},q^{12};q^{12})_\infty}{(q;q)_\infty}=\dfrac{f(-q,-q^{11})}{f_1},\\
\label{Yq} Y(q)&:=\sum_{n=0}^\infty\dfrac{(-q^2;q^2)_n(1+q^n)q^{n^2}}{(q;q)_{2n+2}}=\dfrac{(q^5,q^7,q^{12};q^{12})_\infty}{(q;q)_\infty}=\dfrac{f(-q^5,-q^{7})}{f_1}.
\end{align}

\begin{theorem}\label{Theo12_1}
For integers $k\ge1$ and $m>1$, suppose that
\begin{align*}
\sum_{n=0}^{\infty}\xi_{1,k}(n)q^n&=\dfrac{Y^k(q)}{X^k(q)}, \quad
\sum_{n=0}^{\infty}\xi_{m,k}(n)q^n=\dfrac{X^k(q)Y^k\left(q^m\right)}{Y^k(q)X^k\left(q^m\right)},\\
\sum_{n=0}^{\infty}\xi_{4}(n)q^n&=\dfrac{Y(q)Y\left(q^2\right)}{X(q)X\left(q^2\right)}.
\end{align*}
Then,  for any $n\geq0$, we have
\begin{align*}
&\textup{(i)} & &\xi_{1,1}\left(2n+3\right)=-\xi_{1,1}^{\prime}\left(2n+1\right);&
&\textup{(ii)} & &\xi_{1,1}\left(4n+2\right)=\xi_{1,1}^{\prime}\left(4n\right);\notag\\
&\textup{(iii)} & &\xi_{1,1}\left(6n+2\right)=\xi_{1,1}^{\prime}\left(6n\right);&
&\textup{(iv)} & &\xi_{1,2}\left(12n+7\right)=\xi_{1,2}^{\prime}\left(12n+3\right);\notag\\
&\textup{(v)} & &\xi_{1,2}\left(24n+12\right)=\xi_{1,2}^{\prime}\left(24n+8\right); &
&\textup{(vi)} & &\xi_{2,1}\left(2n+3\right)=-\xi_{2,1}^{\prime}\left(2n+1\right);\\
&\textup{(vii)} & &\xi_{2,1}\left(4n+2\right)=\xi_{2,1}^{\prime}\left(4n\right);&
&\textup{(viii)} & &\xi_{2,1}\left(12n+8\right)=\xi_{2,1}^{\prime}\left(12n+6\right);\\
&\textup{(ix)} & &\xi_{2,2}\left(12n+7\right)=\xi_{2,2}^{\prime}\left(12n+3\right);&
&\textup{(x)} & &\xi_{2,2}\left(24n+12\right)=\xi_{2,2}^{\prime}\left(24n+8\right);\\
&\textup{(xi)} & &\xi_{3,1}\left(3n+5\right)=\xi_{3,1}^{\prime}\left(3n+1\right);&
&\textup{(xii)} & &\xi_{3,1}\left(6n+8\right)=-\xi_{3,1}^{\prime}\left(6n+4\right);\\
&\textup{(xiii)} & &\xi_{3,1}\left(6n+4\right)=-\xi_{3,1}^{\prime}\left(6n\right);&
&\textup{(xiv)} & &\xi_{3,1}\left(12n+7\right)=\xi_{3,1}^{\prime}\left(12n+3\right);\\
&\textup{(xv)} & &\xi_{3,2}\left(24n+14\right)=\xi_{3,2}^{\prime}\left(24n+6\right);&
&\textup{(xvi)} & &\xi_{3,2}\left(36n+25\right)=\xi_{3,2}^{\prime}\left(36n+17\right);\\
&\textup{(xvii)} & &\xi_{4}\left(12n+15\right)=-\xi_{4}^{\prime}\left(12n+9\right);&
&\textup{(xviii)} & &\xi_{4}\left(12n+7\right)=-\xi_{4}^{\prime}\left(12n+1\right);\\
&\textup{(xix)} & &\xi_{4}\left(12n+10\right)=\xi_{4}^{\prime}\left(12n+4\right);&
&\textup{(xx)} & &\xi_{4}\left(12n+11\right)=-\xi_{4}^{\prime}\left(12n+5\right);\\
&\textup{(xxi)} & &\xi_{4}\left(24n+13\right)=\xi_{4}^{\prime}\left(24n+7\right);&
&\textup{(xxii)} & &\xi_{4}\left(36n+24\right)=\xi_{4}^{\prime}\left(36n+18\right).
\end{align*}
\end{theorem}

\begin{remark}
Replacing $q$ by $-q$ in the $q$-products of Theorems \ref{Theo_1}--\ref{Theo12_1}, we can deduce similar results satisfied by the corresponding coefficients. We omit those results.
\end{remark}

The paper is organized as follows. In Section \ref{Preliminary}, we present some preliminary lemmas as well as some new  modular relations involving Rogers-Ramanujan type functions. In Sections \ref{sec3}--\ref{sec7}, we prove Theorems \ref{Theo_1}--\ref{Theo12_1}. In the final section, we pose a few interesting conjectures on matching coefficients.

We end this section by defining the extraction operator $U_{a n+b}$, which acts on a power series as
\begin{align*}
    U_{a n+b}\left(\sum_{n=0}^\infty A(n)q^n\right)&=\sum_{n=0}^\infty A(a n+b)q^n.
\end{align*}
This operator will be used frequently throughout the proofs of the theorems.

\section{Preliminaries}\label{Preliminary}

Some useful results on $f(a,b)$ are presented in the following lemma.

\begin{lemma}\label{ThetaIdentityLemma}
\textup{(\cite[pp. 45--46, Entries 29 and 30 \iffalse-(ii), -(iii), -(v), and -(vi)\fi]{B1})} If $ab=cd$,
\begin{align}
\label{ThetaIdentity0}f(a,b)&=f\left(a^3b,ab^3\right)+a f\left(\dfrac{b}{a},a^5b^3\right),\\
\label{ThetaIdentity1}f^2(-a,-b)&=f\left(a^2,b^2\right)\varphi(ab)-2a f\left(\dfrac{b}{a},a^3b\right)\psi\left(a^2b^2\right),\\
\label{ThetaIdentity2}f(a,b)f(c,d)&+f(-a,-b)f(-c,-d)=2f(ac,bd)f(ad,bc),\\
\label{ThetaIdentity3}f(a,b)f(c,d)&-f(-a,-b)f(-c,-d)=2af\left(\dfrac{b}{c},ac^2d\right)f\left(\dfrac{b}{d},acd^2\right).
\end{align}
\end{lemma}

The next three lemmas contain some known 2-, 3-, and 5-dissections of certain $q$-products.
\begin{lemma}\label{Lemma-2-dissection}\textup{(\cite{BKaur1, Hirsch1,HS1})}
We have
\begin{align*}
&f_1^2=\frac{f_{2} f_{8}^5}{f_{4}^2 f_{16}^2}- 2q\frac{ f_{2} f_{16}^2}{f_{8}}, &
&\dfrac{1}{f_1^2}=\dfrac{f_{8}^5}{f_{2}^5 f_{16}^2}+2q\dfrac{f_{4}^2 f_{16}^2 }{f_{2}^5 f_{8}},\\
&f_1^4=\frac{f_{4}^{10}}{f_{2}^2 f_{8}^4}- 4q\frac{ f_{2}^2 f_{8}^4}{f_{4}^2}, & &\frac{1}{f_1^4}=\frac{f_{4}^{14}}{f_{2}^{14} f_{8}^4}+4q\frac{ f_{4}^2 f_{8}^4 }{f_{2}^{10}},\\
&f_{1} f_{3}=\frac{f_{2} f_{8}^2 f_{12}^4}{f_{4}^2 f_{6} f_{24}^2}-q\frac{f_{4}^4 f_{6} f_{24}^2}{f_{2} f_{8}^2 f_{12}^2},& &\frac{1}{f_1 f_{3}}=\frac{f_{8}^2 f_{12}^5}{f_{2}^2 f_{4} f_{6}^4 f_{24}^2}+q\frac{f_{4}^5 f_{24}^2 }{f_{2}^4 f_{6}^2 f_{8}^2 f_{12}},\\
&\frac{f_{3}}{f_1^3}=\frac{f_{4}^6 f_{6 }^3}{f_{2 }^9 f_{12 }^2}+3 q\frac{ f_{4 }^2 f_{6 } f_{12 }^2}{f_{2 }^7},& &\frac{f_{3}^3}{f_1}=\frac{f_{4 }^3 f_{6 }^2}{f_{2 }^2 f_{12 }}+q\frac{f_{12 }^3 }{f_{4 }},\\
&\frac{f_1}{f_{3}}=\frac{f_{2 } f_{16 } f_{24 }^2}{f_{6 }^2 f_{8 } f_{48 }}-q\frac{f_{2 } f_{8 }^2 f_{12 } f_{48 } }{f_{4 } f_{6 }^2 f_{16 } f_{24 }},& &\frac{f_{3}}{f_1}=\frac{f_{4 } f_{6 } f_{16 } f_{24 }^2}{f_{2 }^2 f_{8 } f_{12 } f_{48 }}+q\frac{f_{6 } f_{8 }^2 f_{48 }}{f_{2 }^2 f_{16 } f_{24 }},\\
&\frac{f_1^2}{f_{3}^2}=\frac{f_{2 } f_{4 }^2 f_{12 }^4}{f_{6 }^5 f_{8 } f_{24 }}-2q\frac{f_{2 }^2 f_{8 } f_{12 } f_{24 }}{f_{4 } f_{6 }^4},& &\frac{f_1}{f_{5}}=\frac{f_{2 } f_{8 } f_{20 }^3}{f_{4 } f_{10 }^3 f_{40 }}- q\frac{f_{4 }^2 f_{40 }}{f_{8 } f_{10 }^2},\\
&\frac{f_{5}}{f_1}=\frac{f_{8 } f_{20 }^2}{f_{2 }^2 f_{40 }}+q\frac{f_{4 }^3 f_{10 } f_{40 } }{f_{2 }^3 f_{8 } f_{20 }}.
\end{align*}
\end{lemma}

\begin{lemma}\label{Lemma-3-dissection} \textup{(\cite{Hirsch1})}
We have
\begin{align*}
&\frac{f_1^2}{f_{2}}=\frac{f_{9 }^2}{f_{18 }}-2 q\frac{f_{3 } f_{18 }^2}{f_{6 } f_{9 }},& &\frac{f_{2}}{f_1^2}=\frac{f_{6 }^4 f_{9 }^6}{f_{3 }^8 f_{18 }^3}+2 q\frac{f_{6 }^3 f_{9 }^3}{f_{3 }^7}+4 q^{2 }\frac{f_{6 }^2 f_{18 }^3}{f_{3 }^6},\\
&\frac{f_1 f_{4}}{f_{2}}=\frac{f_{3 } f_{12 } f_{18 }^5}{f_{6 }^2 f_{9 }^2 f_{36 }^2}-q\frac{f_{9 } f_{36 }}{f_{18 }}, & &\frac{f_{2}}{f_1 f_{4}}=\frac{f_{18 }^9}{f_{3 }^2 f_{9 }^3 f_{12 }^2 f_{36 }^3}+q\frac{f_{6 }^2 f_{18 }^3}{f_{3 }^3 f_{12 }^3}+q^{2 }\frac{f_{6 }^4 f_{9 }^3 f_{36 }^3}{f_{3 }^4 f_{12 }^4 f_{18 }^3},\\
&f_1^3=\dfrac{f_6f_9^6}{f_3f_{18}^3}-3qf_9^3+4q^3 \dfrac{f_3^2f_{18}^6}{f_6^2f_9^3},& &f_1 f_{2}=\frac{f_{6 } f_{9 }^4}{f_{3 } f_{18 }^2}-qf_{9 } f_{18 }-2 q^{2 }\frac{f_{3 } f_{18 }^4}{f_{6 } f_{9 }^2}.
\end{align*}
\end{lemma}

\begin{lemma}\label{5dissectf1phipsi}\textup{(\cite[p. 80, Entry 38(iv) and p. 49, Corollary]{B1})}
We have
\begin{align}
\label{5df_1}
f_1&=f_{25}\left(\dfrac{1}{R\left(q^5\right)}-q-q^2R\left(q^5\right)\right),\\
\label{5dissectPhi}\varphi(q)&=\varphi\left(q^{25}\right)+2qf\left(q^{15},q^{35}\right)+2q^{4}f\left(q^{5},q^{45}\right),\\
\label{5dissectPsi}\psi(q)&=f\left(q^{10},q^{15}\right)+qf\left(q^{5},q^{20}\right)+q^3\psi\left(q^{25}\right).
\end{align}
\end{lemma}

 The first three identities in the following lemma  are in the list of the forty identities recorded by Ramanujan  \cite{memoirs}. The remaining two identities of the lemma were found by Robins \cite[Chapter 1, (1.25), (1.26)]{R1}.
\begin{lemma}\label{Rama40Lemma}
We have
\begin{align}
\label{Rama1}G\left(q\right)G\left(q^4\right)-qH\left(q\right)H\left(q^4\right)&=\dfrac{f_{10}^5}{f_2f_5^2f_{20}^2},\\
\label{Rama2}G\left(-q\right)G\left(-q^4\right)+qH\left(-q\right)H\left(-q^4\right)&=\dfrac{f_4^2}{f_2f_8},\\
\label{Rama3}G\left(q^{16}\right)H\left(q\right)-q^3G\left(q\right)H\left(q^{16}\right)&=\dfrac{f_4^2}{f_2f_8},\\
\label{RobinsKey1}G^2\left(q\right)H\left(q^2\right) - G\left(q^2\right)H^2\left(q\right)&=2qH\left(q\right)H^2\left(q^2\right)\dfrac{f_{10}^2}{f_5^2},\\
\label{RobinsKey2}G^2\left(q\right)H\left(q^2\right) +G\left(q^2\right)H^2\left(q\right)&=2G\left(q\right)G^2\left(q^2\right)\dfrac{f_{10}^2}{f_5^2}.
\end{align}
\end{lemma}

In the next lemma we present three new identities for $G(q)$ and $H(q)$.
\begin{lemma}\label{NewLemmaGH}
We have
\begin{align}
\label{newGH1}G^3\left(q^2\right)H\left(q\right)+qG\left(q\right)H^3\left(q^2\right)&=\dfrac{f_2 f_{10}^9}{f_1 f_4 f_5^5 f_{20}^3}+4q^2\dfrac{f_4 f_{20}^3 }{f_2^2 f_5^2},\\
\label{newGH2}G\left(q^{16}\right)H\left(q\right)+q^3G\left(q\right)H\left(q^{16}\right)&=\dfrac{f_{20}^2}{f_2f_{40}}+2q^3\dfrac{f_8f_{80}}{f_2f_{16}},\\
\label{newGH3}G\left(-q\right)G\left(-q^4\right)-qH\left(-q\right)H\left(-q^4\right)&=\dfrac{f_{20}^2}{f_2f_{40}}-2q\dfrac{f_4f_{16}f_{40}^3}{f_2f_8^2f_{20}f_{80}}.
\end{align}
\end{lemma}

\begin{proof}
Setting $a=-q,b=-q^4$ and $a=-q^2,b=-q^3$, in turn,  in \eqref{ThetaIdentity1}, we have
\begin{align}
\label{NewGH1Key11}f^2\left(q,q^4\right)&=f\left(q^2,q^8\right)\varphi(q^5)+2q f\left(q^3,q^7\right)\psi\left(q^{10}\right),\\
\label{NewGH1Key21}f^2\left(q^2,q^3\right)&=f\left(q^4,q^6\right)\varphi(q^5)+2q^2 f\left(q,q^9\right)\psi\left(q^{10}\right).
\end{align}
Multiplying \eqref{NewGH1Key11} by $qf\left(-q,-q^9\right)$ and \eqref{NewGH1Key21} by $f\left(-q^3,-q^7\right)$, and then adding the resulting identities, we find that
\begin{align}
\label{ModularKey1}f^2\left(q^2,q^3\right)&f\left(-q^3,-q^7\right)+qf^2\left(q,q^4\right)f\left(-q,-q^9\right)\notag\\
&=\varphi(q^5)\left(f\left(q^4,q^6\right)f\left(-q^3,-q^7\right)+qf\left(q^2,q^8\right)f\left(-q,-q^9\right)\right)\notag\\
&\quad+2q^2\psi\left(q^{10}\right)\left(f\left(-q^3,-q^7\right)f\left(q,q^9\right)+f\left(q^3,q^7\right)f\left(-q,-q^9\right)\right).
\end{align}

Now, setting $a=q,b=-q^4,c=-q^2,$ and $d=q^3$  in \eqref{ThetaIdentity2} and \eqref{ThetaIdentity3} and then adding the resulting identities, we obtain
\begin{align*}
f\left(q^4,q^6\right)f\left(-q^3,-q^7\right)+qf\left(q^2,q^8\right)f\left(-q,-q^9\right)&=f\left(q,-q^4\right)f\left(-q^2,q^3\right),
\end{align*}
which, by the Jacobi triple product identity \eqref{jtpi}, reduces to
\begin{align}\label{gh1}
f\left(q^4,q^6\right)f\left(-q^3,-q^7\right)+qf\left(q^2,q^8\right)f\left(-q,-q^9\right)=f(q)f(q^5).
\end{align}

On the other hand, setting $a=q,b=q^9,c=-q^3,$ and $d=-q^7$  in \eqref{ThetaIdentity2}, and then using \eqref{jtpi}, we have
\begin{align}\label{gh2}
f\left(-q^3,-q^7\right)f\left(q,q^9\right)+f\left(q^3,q^7\right)f\left(-q,-q^9\right)&=2f_{4}f_{20}.
\end{align}
Employing \eqref{gh1} and \eqref{gh2} in \eqref{ModularKey1}, we arrive at
\begin{align}
\label{QQQ}f^2\left(q^2,q^3\right)f\left(-q^3,-q^7\right)+qf^2\left(q,q^4\right)f\left(-q,-q^9\right)&=f(q)f(q^5)\varphi\left(q^5\right)\notag\\
&\quad+4q^2f_4f_{20}\psi\left(q^{10}\right).
\end{align}

Now, using \eqref{jtpi} and the $q$-product representations of $G(q)$ and $H(q)$, we find that
\begin{align*}
&f\left(q,q^4\right)=f_5\dfrac{G(q)}{G\left(q^2\right)}, \quad f\left(-q,-q^9\right)=f_{10}\dfrac{H\left(q^2\right)}{G\left(q\right)},\quad f\left(q^2,q^3\right)=f_{5}\dfrac{H\left(q\right)}{H\left(q^2\right)}, \notag\\
&f\left(-q^3,-q^7\right)=f_{10}\dfrac{G\left(q^2\right)}{H\left(q\right)}.
\end{align*}
Employing these above identities in \eqref{QQQ} and also noting that
\begin{align}\label{GHf5}
G\left(q\right)H\left(q\right)=\dfrac{f_5}{f_1},\\
\label{fq}f(q)=\dfrac{f_2^3}{f_1f_4},
\end{align}we obtain
\begin{align*}
G^3\left(q^2\right)H\left(q\right)+qG\left(q\right)H^3\left(q^2\right)=\dfrac{f_2 f_{10}^9}{f_1 f_4 f_5^5 f_{20}^3}+4q^2\dfrac{f_4 f_{20}^3 }{f_2^2 f_5^2}.
\end{align*} Thus, we complete the proof of \eqref{newGH1}.

Next, we recall the 2-dissections of $G(q)$ and $H(q)$ due to Watson \cite{W1}, namely,
\begin{align}\label{2-dissectGH}
G\left(q\right)=\dfrac{f_8}{f_2}\left(G\left(q^{16}\right)+qH\left(-q^4\right)\right),\quad H\left(q\right)=\dfrac{f_8}{f_2}\left(G\left(-q^4\right)+q^3H\left(q^{16}\right)\right).
\end{align}
Therefore,
\begin{align}
\label{Modular5Key2}G\left(q^{16}\right)H\left(q\right)+q^3G\left(q\right)H\left(q^{16}\right)&=\dfrac{f_8}{f_2}\bigg(2q^3G\left(q^{16}\right)H\left(q^{16}\right)+G\left(-q^4\right)G\left(q^{16}\right)\notag\\
&\quad+q^4H\left(-q^4\right)H\left(q^{16}\right)\bigg).
\end{align}

Now, replacing $q$ by $-q^4$ in \eqref{Rama1}, we have
\begin{align}
\label{M2}G\left(-q^4\right)G\left(q^{16}\right)+q^4H\left(-q^4\right)H\left(q^{16}\right)=\dfrac{f_{40}^5}{f_8f_{80}^2f^2(q^{20})}=\dfrac{f_{20}^2}{f_8f_{40}},
\end{align}
where we  use \eqref{fq} in the last equality. Employing \eqref{GHf5} and \eqref{M2} in \eqref{Modular5Key2}, we arrive at \eqref{newGH2}.

Next, with the aid of \eqref{2-dissectGH} with $q$ replaced by $-q$, we have
\begin{align}\label{M3}
G\left(-q\right)G\left(-q^4\right)-qH\left(-q\right)H\left(-q^4\right)&=\dfrac{f_8}{f_2}\bigg(G\left(-q^4\right)G\left(q^{16}\right)+q^4H\left(-q^4\right)H\left(q^{16}\right)\notag\\
&\quad-2qG\left(-q^{4}\right)H\left(-q^{4}\right)\bigg).
\end{align}
Employing \eqref{M2}, \eqref{GHf5}, and \eqref{fq} in \eqref{M3}, we arrive at \eqref{newGH3}.
\end{proof}

The first identity in the next lemma was communicated by Ramanujan in his first letter to Hardy whereas the remaining two   were proved by Gugg \cite[Theorem 3.1]{G1}.
\begin{lemma}\label{GuggLemma} We have
\begin{align}\label{R5q5}
\dfrac{R^5\left(q\right)}{R\left(q^5\right)}=\dfrac{1-2 q R\left(q^5\right)+4 q^2 R^2\left(q^5\right)-3 q^3 R^3\left(q^5\right)+q^4 R^4\left(q^5\right)}{1+3 q R\left(q^5\right)+4 q^2 R^2\left(q^5\right)+2 q^3 R^3\left(q^5\right)+q^4 R^4\left(q^5\right)},\\
\label{R5q5N}1-2 q R\left(q^5\right)+4 q^2 R^2\left(q^5\right)-3 q^3 R^3\left(q^5\right)+q^4 R^4\left(q^5\right)=R^2\left(q^5\right)\dfrac{H^5(q)f_1^2}{H\left(q^5\right)f_{25}^2},\\
\label{R5q5D}1+3 q R\left(q^5\right)+4 q^2 R^2\left(q^5\right)+2 q^3 R^3\left(q^5\right)+q^4 R^4\left(q^5\right)=R^2\left(q^5\right)\dfrac{G^5(q)f_1^2}{G\left(q^5\right)f_{25}^2}.
\end{align}
\end{lemma}

In the next lemma, we present a couple of identities connecting $S(q)$ and $T(q)$ defined in \eqref{Sq} and \eqref{Tq}.
\begin{lemma}\label{LemmaSTModular}
 We have
\begin{align}
\label{ModularIdentityST1}S\left(-q\right)T\left(q\right)S\left(q^4\right)-q^2  S\left(q\right)T\left(-q\right)T\left(q^4\right)&= \dfrac{f_1f_{16}^3}{f_2f_4f_8f_{32}},\\
\label{ModularIdentityST2}S\left(-q\right)T\left(q\right)S\left(q^4\right)+q^2  S\left(q\right)T\left(-q\right)T\left(q^4\right)&= \dfrac{f_1 f_8^4f_{16} }{f_2^3 f_4^2 f_{32}}+4 q^3\dfrac{ f_1 f_{32}^3 }{f_2^3 f_{16}}.
\end{align}
\end{lemma}

\begin{proof}
Identity \eqref{ModularIdentityST1} was found by Xia and Yao \cite[(2.7)]{XY1} whereas \eqref{ModularIdentityST2} seems to be new. So, we  prove only \eqref{ModularIdentityST2}. Using \eqref{jtpi} and manipulating the $q$-products, we have
\begin{align}
\label{STf} S(q)T(q)=\dfrac{f_2 f_8^2}{f_1 f_4^2}, ~f(-q,-q^{7})=\dfrac{f_{8}^2}{S(q)f_4}, ~f(-q^3,-q^{5})=\dfrac{f_{8}^2}{T(q)f_4}.
\end{align}

Setting $a=q^3,b= q^5,c=-q^3,$ and $d=-q^5$ in \eqref{ThetaIdentity2}, we find that
\begin{align*}
f(-q^3,-q^{5})f(q^3,q^{5})&=f(-q^6,-q^{10})f(-q^8,-q^{8}).
\end{align*}
Multiplying both sides by $f(q^6,q^{10})$ and once again using \eqref{ThetaIdentity2} with $a=q^6,$ $b= q^{10}$, $c=-q^6,$ and $d=-q^{10}$, we obtain
\begin{align}\label{q35}
f(-q^3,-q^{5})f(q^3,q^{5})f(q^6,q^{10})&=f(-q^{12},-q^{20})f(-q^8,-q^{8})f(-q^{16},-q^{16}).
\end{align}
In a similar way, we find that
\begin{align}\label{q17}
f(-q,-q^{7})f(q,q^{7})f(q^2,q^{14})&=f(-q^{4},-q^{28})f(-q^8,-q^{8})f(-q^{16},-q^{16}).
\end{align}
With the help of \eqref{q35} and \eqref{q17}, we have
\begin{align}
\label{Modular8Key1}&f(-q,-q^{7})f(q^3,q^{5})f(-q^{12},-q^{20})+q^2f(q,q^{7})f(-q^3,-q^{5})f(-q^{4},-q^{28})\notag\\
&=\dfrac{f(-q,-q^{7})f(-q^3,-q^{5})}{f(-q^8,-q^{8})f(-q^{16},-q^{16})}\Big(f^2(q^3,q^{5})f(q^6,q^{10})+q^2f^2(q,q^{7})f(q^2,q^{14})\Big).
\end{align}

Next, setting $a=-q,b=-q^7$ and  $a=-q^3,b=-q^5$, in turn, in \eqref{ThetaIdentity1}, we have
\begin{align}
\label{Multi-q} f^2(q,q^{7})&=f(q^2,q^{14})\varphi(q^8)+2q f(q^6,q^{10})\psi(q^{16}),\\
\label{ThenAdd}f^2(q^3,q^{5})&=f(q^6,q^{10})\varphi(q^8)+2q^3 f(q^2,q^{14})\psi(q^{16}).
\end{align}
With the help of the above identities, we find that
\begin{align}
\label{Reduce1} f^2(q^3,q^{5})+qf^2(q,q^{7})=\left(f(q^6,q^{10})+qf(q^2,q^{14})\right)\left(\varphi(q^8)+2q^2\psi(q^{16})\right).
\end{align}
But, setting $a=q$ and $b=q$ and $a=q$ and $b=q^3$, in turn, in  \eqref{ThetaIdentity0}, we have
\begin{align}\label{phi-2}
\varphi(q)&=\varphi\left(q^4\right)+qf(1, q^8)= \varphi\left(q^4\right)+2q\psi\left(q^{8}\right),\\
\label{psi-2}\psi(q)&=f(q^6,q^{10})+qf(q^2,q^{14}).
\end{align}
Using \eqref{phi-2} with $q$ replaced by $q^2$  and \eqref{psi-2} in \eqref{Reduce1}, and then employing the product representations of $\varphi(q^2)$ and $\psi(q)$, we find that
\begin{align}
\label{Reduced} f^2(q^3,q^{5})+qf^2(q,q^{7})=\dfrac{f_4^5}{f_1f_8^2}.
\end{align}

Now, from \eqref{Multi-q} and \eqref{ThenAdd}, we have
\begin{align}
\label{Modular8Key2}&f^2\left(q^3,q^{5}\right)f\left(q^6,q^{10}\right)+q^2f^2\left(q,q^{7}\right)f\left(q^2,q^{14}\right)\notag\\
&=\left(f^2\left(q^6,q^{10}\right)+q^2f^2\left(q^2,q^{14}\right)\right)\varphi\left(q^8\right)+4q^3f\left(q^2,q^{14}\right)f\left(q^6,q^{10}\right)\psi\left(q^{16}\right).
\end{align}
Employing \eqref{Reduced}, \eqref{jtpi}, the product representations of $\varphi(q)$ and $\psi(q)$, and then manipulating the $q$-products, we find from \eqref{Modular8Key1} and \eqref{Modular8Key2}  that
\begin{align*}
f\left(-q,-q^{7}\right)f\left(q^3,q^{5}\right)f\left(-q^{12},-q^{20}\right)&+q^2f\left(q,q^{7}\right)f\left(-q^3,-q^{5}\right)f\left(-q^{4},-q^{28}\right)\\
&=\frac{f_1 f_8^3 f_{16}^2}{f_2^2 f_{32}}+4 q^3\frac{f_1 f_4^2 f_{32}^3}{f_2^2 f_8},
\end{align*}
which, with the aid of \eqref{STf}, can be seen to be equivalent to \eqref{ModularIdentityST2}. Thus, we complete the proof.
\end{proof}

We end this section by  presenting some identities connecting $X(q)$ and $Y(q)$ defined in \eqref{Xq} and \eqref{Yq}.
\begin{lemma}\label{NewlemmaXY}
 We have
\begin{align}
\label{BBXY1}X\left(q\right)Y\left(q^3\right)+q^2 X\left(q^3\right)Y\left(q\right)&= \frac{f_4 f_6^5 f_9 f_{36}}{f_2^2 f_3^3 f_{12}^2 f_{18}},\\
\label{RobinsXY1} X\left(q\right)Y\left(q^3\right)-q^2 X\left(q^3\right)Y\left(q\right)&= \frac{f_{18}^2}{f_3 f_9},\\
\label{newXY1}X\left(q\right) Y\left(q^2\right)-q X\left(q^2\right) Y\left(q\right)&= \frac{f_1 f_6 f_{24}}{f_2^2 f_3},\\
\label{newXY2}Y\left(q^2\right) Y\left(q\right)-q^3 X\left(q^2\right) X\left(q\right)&=\frac{f_2^3 f_8 f_{12}^5}{f_1^2 f_4^3 f_6^2 f_{24}^2}- q\frac{f_1 f_6 f_{24}}{f_2^2 f_3}.
\end{align}
\end{lemma}
\begin{proof} The first two identities are due to Baruah and Bora \cite{BB1} and Robins \cite{R1}, so we prove  \eqref{newXY1} and \eqref{newXY2} only, which seem to be new.
Setting $a=q^5,b= q^7,c=-q^5,$ and $d=-q^7$ and $a=q,b= q^{11},c=-q,$ and $d=-q^{11}$, in turn, in \eqref{ThetaIdentity2}, we find that
\begin{align*}
f(q^5, q^7)f(-q^5,-q^7)&=f(-q^{10}, -q^{14})f(-q^{12}, -q^{12}),\\
f(q, q^{11})f(-q,-q^{11})&=f(-q^{2}, -q^{22})f(-q^{12}, -q^{12}).
\end{align*}
Therefore,
\begin{align}
\label{Modular12Key1}&f\left(-q,-q^{11}\right)f\left(-q^{10},-q^{14}\right)-qf\left(-q^5,-q^{7}\right)f\left(-q^2,-q^{22}\right)\notag\\
&=\dfrac{f\left(-q,-q^{11}\right)f\left(-q^{5},-q^{7}\right)}{f(-q^{12}, -q^{12})}\left(f\left(q^{5},q^{7}\right)-q f\left(q,q^{11}\right)\right).
\end{align}
As $a=-q$ and $b=-q^2$ in \eqref{ThetaIdentity0} gives
\begin{align}\label{Modular12Key11}
f\left(q^{5},q^{7}\right)-q f\left(q,q^{11}\right)=f\left(-q,-q^{2}\right)=f_1,
\end{align}
it follows from \eqref{Modular12Key1} that
\begin{align*}
f\left(-q,-q^{11}\right)f\left(-q^{10},-q^{14}\right)&-qf\left(-q^5,-q^{7}\right)f\left(-q^2,-q^{22}\right)\\
&=f_1
\dfrac{f\left(-q,-q^{11}\right)f\left(-q^{5},-q^{7}\right)}{f(-q^{12}, -q^{12})}.
\end{align*}
Employing \eqref{Xq}, \eqref{Yq}, and \eqref{jtpi} in the above identity, and then simplifying the $q$-products, we arrive at \eqref{newXY1}.

We now prove \eqref{newXY2}. We have
\begin{align}\label{xy12}
&f(-q^5,-q^{7})f(-q^{10},-q^{14})-q^3f(-q,-q^{11})f(-q^2,-q^{22})\notag\\
&=\left(f(-q^5,-q^{7})+qf(-q,-q^{11})\right)\cdot \left(f(-q^{10},-q^{14})-q^2f(-q^2,-q^{22})\right)\notag\\
&\quad-q\left(f(-q,-q^{11})f(-q^{10},-q^{14})-qf(-q^5,-q^{7})f(-q^2,-q^{22})\right).
\end{align}

Setting $a=-q^2$ and $b=q^4$ in \eqref{ThetaIdentity0}, we have
\begin{align}\label{Modular12Key121}
f(-q^2,q^4)=f(-q^{10},-q^{14})-q^2 f(-q^2,-q^{22}).
\end{align}
Employing \eqref{Xq}, \eqref{Yq}, \eqref{jtpi}, \eqref{newXY1}, \eqref{Modular12Key11} with $q$ replaced by $-q$,     and \eqref{Modular12Key121} in \eqref{xy12}, and then simplifying the $q$-products, we obtain \eqref{newXY2} to complete the proof.
\end{proof}

\section{Proofs of Theorems \ref{Theo_1} and \ref{ExtraTheo1}}\label{sec3}
\textit{Proof of Theorem \ref{Theo_1}} We recall from Cui and Gu \cite[Theorem 2.2]{CG1} that for a prime $p>3$
\begin{align}
\label{pdissectf1} f_1&=(-1)^{\frac{\pm p-1}{6}}q^{\frac{p^2-1}{24}}f_{p^2}+\sum_{\substack{k=-\frac{p-1}{2}\\ k\ne \frac{\pm p-1}{6}}}^{\frac{p-1}{2}}(-1)^k q^{\frac{3k^2+k}{2}}f\left(-q^{\frac{3p^2+(6k+1)p}{2}},-q^{\frac{3p^2-(6k+1)p}{2}}\right),
\end{align}
where
\begin{align*}
\dfrac{\pm p-1}{6}:=\left\{
                      \begin{array}{ll}
                        \frac{p-1}{6}, & \hbox{$p\equiv1~(\text{mod}~6)$,} \\
                      \frac{-p-1}{6}, & \hbox{$p\equiv-1~(\text{mod}~6)$.}
                      \end{array}
                    \right.
\end{align*}

Furthermore, for $-(p-1)/2\le k \le (p-1)/2$ and $k\ne (\pm p-1)/6$,
\begin{align}\label{condition1}
\dfrac{3k^2+k}{2}\not\equiv \dfrac{p^2-1}{24}~(\text{mod}~p).
\end{align}

Dividing both sides of \eqref{pdissectf1} by  $f_p$, we find that
\begin{align*}
 \left(q,q^2,\ldots,q^{p-1};q^{p}\right)_\infty&=(-1)^{\frac{\pm p-1}{6}}q^{\frac{p^2-1}{24}}\left(q^p,q^{2p},\ldots,q^{p(p-1)};q^{p^2}\right)_\infty^{-1}\notag\\
&\quad+\dfrac{1}{f_p}\sum_{\substack{k=-\frac{p-1}{2}\\ k\ne \frac{\pm p-1}{6}}}^{\frac{p-1}{2}}(-1)^k q^{\frac{3k^2+k}{2}}f\left(-q^{\frac{3p^2+(6k+1)p}{2}},-q^{\frac{3p^2-(6k+1)p}{2}}\right),
\end{align*}
Applying $U_{pn+(p^2-1)/24}$ on both sides of the above identity and then noting the condition \eqref{condition1}, we arrive at
\begin{align*}
U_{pn+(p^2-1)/24}\left(\sum_{n=0}^{\infty}\gamma_1(n)q^n\right)&=(-1)^{\frac{\pm p-1}{6}}\left(q,q^2,\ldots,q^{p-1};q^{p}\right)_\infty^{-1}\\
&=(-1)^{\frac{\pm p-1}{6}}\sum_{n=0}^{\infty}\gamma_1^{\prime}(n)q^n,
\end{align*}
from which (i) follows.

Next, we prove (iii). From Cui and Gu \cite[Theorem 2.1]{CG1}, we recall the $p$-dissection of $\psi(q)$ for any odd prime $p$ as
\begin{align}
\label{pdissectPsi}
\psi(q)=q^{\frac{p^2-1}{8}}\psi\left(q^{p^2}\right)+\sum_{k=0}^{\frac{p-3}{2}}q^{\frac{k^2+k}{2}}f\left(q^{\frac{p^2+(2k+1)p}{2}},q^{\frac{p^2-(2k+1)p}{2}}\right).
\end{align}
Furthermore, for $0\le k\le (p-3)/2$,
\begin{align}
\label{condition3}\dfrac{k^2+k}{2}\not\equiv  \dfrac{p^2-1}{8}~(\text{mod}~p).
\end{align}

Dividing both sides of \eqref{pdissectPsi} by  $\psi(q^p)={f_{2p}^2}/{f_{p}}$, we find that
\begin{align*}
\dfrac{\left(q^2,q^4,\ldots,q^{2(p-1)};q^{2p}\right)_{\infty}^{2}}{\left(q,q^2,\ldots,q^{p-1};q^{p}\right)_{\infty}}
&=q^{\frac{p^2-1}{8}}\dfrac{\left(q^p,q^{2p},\ldots,q^{p(p-1)};q^{p^2}\right)_{\infty}}{\left(q^{2p},q^{4p},\ldots,q^{2p(p-1)};q^{2p^2}\right)_{\infty}^{2}}\\
&\quad+\dfrac{f_{p}}{f_{2p}^2}\sum_{k=0}^{\frac{p-3}{2}}q^{\frac{k^2+k}{2}}f\left(q^{\frac{p^2+(2k+1)p}{2}},q^{\frac{p^2-(2k+1)p}{2}}\right).
\end{align*}
Using $U_{pn+(p^2-1)/8}$ on both sides of the identity above and then considering the condition \eqref{condition3}, we find that
\begin{align*}
U_{pn+(p^2-1)/8}\left(\sum_{n=0}^{\infty}\gamma_3(n)q^n\right)=\dfrac{\left(q,q^2,\ldots,q^{p-1};q^{p}\right)_{\infty}}{\left(q^2,q^4,\ldots,q^{2(p-1)};q^{2p}\right)_{\infty}^{2}}
=\sum_{n=0}^{\infty}\gamma_3^{\prime}(n)q^n,
\end{align*}
from which, (iii) follows readily.

Identity (ii) can be proved in a similar fashion by applying the $p$-dissection of $\varphi(q)$ for a prime $p$ from \cite[Eq. (3.32)]{BK1}, namely,
\begin{align*}
 \varphi(q)&=\varphi\left(q^{p^2}\right)+\sum_{k=1}^{p-1}q^{k^2} f\left(q^{p(p-2k)},q^{p(p+2k)}\right).
\end{align*}
Note that though the above identity is stated  for a prime $p\equiv2~(\text{mod}~3)$  in \cite{BK1} but it also holds for $p=3$ and any prime $p\equiv1~(\text{mod}~3)$.

\vspace*{2mm}
\noindent \textit{Proof of Theorem \ref{ExtraTheo1}}
From Lemma \ref{Lemma-3-dissection}, we recall the 3-dissections of $f_1^3$ and $f_1f_2$, namely,
\begin{align*}
	 f_1^3&=\dfrac{f_6f_9^6}{f_3f_{18}^3}-3qf_9^3+4q^3 \dfrac{f_3^2f_{18}^6}{f_6^2f_9^3},\\
    f_1f_2&=\frac{f_{6} f_{9}^4}{f_{3} f_{18}^2}-qf_{9} f_{18}-2 q^{2}\frac{f_{3} f_{18}^4}{f_{6}f_{9}^2}.
  \end{align*}
Therefore,
\begin{align}
	\label{3f_1a} \dfrac{f_1^3}{f_3^3}&=\dfrac{f_6f_9^6}{f_3^4f_{18}^3}-3q\dfrac{f_9^3}{f_3^3}+4q^3 \dfrac{f_{18}^6}{f_3f_6^2f_9^3},\\
    \label{3f_1f_2a} \dfrac{f_1f_2}{f_3f_6}&=\frac{f_{9}^4}{f_3^2 f_{18}^2}-q\dfrac{f_{9} f_{18}}{f_3f_6}-2 q^{2}\frac{f_{18}^4}{f_6^2f_{9}^2}.
  \end{align}
Applying the operator $U_{3n+1}$ in \eqref{3f_1a} and \eqref{3f_1f_2a}, we readily arrive at (i) and (xi), respectively.

On the other hand, squaring both \eqref{3f_1a} and \eqref{3f_1f_2a}, and then applying the operator  $U_{3n+2}$, we obtain  (ii) and (xii). Similarly, taking the fifth power on both sides of \eqref{3f_1f_2a}, we can easily arrive at (xiii).

The remaining identities of Theorem \ref{ExtraTheo1} can also be proved easily in a similar fashion. For example, to prove (xx), we use $p=11$ in \eqref{pdissectf1} to obtain the $11$-dissections of $f_1$ and $f_4$. We use those to arrive at
\begin{align*}
U_{7n+3}\left(\dfrac{f_1f_4}{f_{11}f_{44}}\right)=q\dfrac{f_{11}f_{44}}{f_1f_4},
\end{align*}
which readily gives (xx). \qed

\section{Proof of Theorem \ref{ExtraTheorem1}} \label{sec4}

The results in  Theorem \ref{ExtraTheorem1} follow from modular equations. For example, a degree 5 modular equation recorded by Ramanujan in his notebooks \cite{nb} and  proved by Berndt \cite[Entry 13(ii), p. 280]{B1}, namely,
\begin{align*}
\left(\dfrac{\alpha^5}{\beta}\right)^{1/8}-\left(\dfrac{(1-\alpha)^5}{1-\beta}\right)^{1/8}=1+2^{1/3}\left(\dfrac{\alpha^5(1-\alpha)^5}{\beta(1-\beta)}\right)^{1/24},
\end{align*}
where $\beta$ has degree 5 over $\alpha$, can be transcribed into the $q$-product identity (See Baruah and Berndt \cite[Eq. (7.4)]{BBerndt1})
\begin{align}
\label{ana1}\dfrac{\left(q;q^2\right)_{\infty}^5}{\left(q^5;q^{10}\right)_{\infty}}+\dfrac{\left(-q;q^2\right)_{\infty}^5}{\left(-q^5;q^{10}\right)_{\infty}}+2
=4\dfrac{\left(q^{10};q^{20}\right)_{\infty}}{\left(q^2;q^{4}\right)_{\infty}^5}.
\end{align}
Equivalently, we have
\begin{align*}
\sum_{n=0}^{\infty}\nu_{1,5}(n)q^n+\sum_{n=0}^{\infty}\nu_{1,5}(n)(-q)^n+2=4\sum_{n=0}^{\infty}\nu_{1,5}^{\prime}(n)q^{2n}.
\end{align*}
Equating the coefficients of $q^{2n+2}$ from both sides, we immediately arrive at (v).

The remaining identities of the theorem can be proved in a similar fashion. In the following list, we  cite only the sources of the used $q$-product identities.
\begin{align*}
\begin{tabular}{c c}
 \hline
  Results in Theorem \ref{ExtraTheorem1}  & Used identities analogous to \eqref{ana1} \T\B\\  \hline
    (i)--(iv), (vi) & \cite[(4.3), (6.5), (7.5)]{BBerndt1}\T\B\\
(vii), (viii) & \cite[(4.8), (4.12)]{BBoruah1}, \cite[(8.15)]{BBS} \T\B\\
 (ix)--(xiv) & \cite[(3.4), (4.4), (5.4), (6.4), (7.4))]{B2} \T\B\\
  (xv) & \cite[(8.4)]{BBerndt1} \T\B\\
  (xvi), (xvii) & \cite[(3.21)]{BBoruah2} (Also \eqref{5df_1}) \T\B\\\hline
\end{tabular}
\end{align*}
\qed

\section{Proof of Theorem \ref{Theo5_1}}\label{sec5}

\noindent\textit{Proofs of (i) and (ii)} Set
\begin{align*}
\Lambda_1:=\sum_{n=0}^{\infty}\lambda_{1}(n)q^n+q^2\sum_{n=0}^{\infty}\lambda^{\prime}_{1}(n)q^n=\dfrac{1}{R^5\left(q\right)}+q^2 R^5\left(q\right).
\end{align*}
Employing \eqref{R5q5}--\eqref{R5q5D} and \eqref{GHf5}, we find that
\begin{align*}
\Lambda_1&=\dfrac{f_1f_{25}^5}{R\left(q^5\right)f_5^6}\Big(1+6 q R\left(q^5\right)+18 q^2 R^2\left(q^5\right)+24 q^3 R^3\left(q^5\right)+42 q^4 R^4\left(q^5\right)\notag\\
&\quad+42 q^6 R^6\left(q^5\right)-24 q^7 R^7\left(q^5\right)+18 q^8 R^8\left(q^5\right)-6 q^9 R^9\left(q^5\right)+q^{10} R^{10}\left(q^5\right)\Big),
\end{align*}
which by  \eqref{5df_1} becomes
\begin{align*}
\Lambda_1&=\dfrac{f_{25}^6}{f_5^6}\bigg(\frac{1}{R^6\left(q^5\right)}+\frac{5q}{R^5(q^5)}+\frac{11q^2}{R^4(q^5)}-\frac{66q^5}{R(q^5)}-66 q^7 R(q^5)\notag\\
&\quad-11 q^{10} R^4(q^5)+5 q^{11} R^5(q^5)-q^{12} R^6(q^5)\bigg).
\end{align*}
Now it is easy to see that $U_{5n+r}(\Lambda_1)=0$ for $r\in\{3,4\}$, which is equivalent to (i). The proof of (ii) can similarly be accomplished.

\noindent\textit{Proof of (iii)} Set
\begin{align*}
\Lambda_3:=\sum_{n=0}^{\infty}\lambda_{3}(n)q^n+\sum_{n=0}^{\infty}\lambda^{\prime}_{3}(n)q^n=\dfrac{R^2(q)}{R(q^2)}+\dfrac{R(q^2)}{R^2(q)}.
\end{align*}

From \eqref{RobinsKey1} and \eqref{RobinsKey2}, we have
\begin{align*}
\dfrac{G^2(q)H(q^2)}{G(q^2)H(q)^2}-1=2q\dfrac{H^2(q^2)}{G(q^2)H(q)}\dfrac{f_{10}^2}{f_5^2}, \quad \dfrac{G(q^2)H^2(q)}{G^2(q)H(q^2)}+1=2\dfrac{G^2(q^2)}{G(q)H(q^2)}\dfrac{f_{10}^2}{f_5^2}.
\end{align*}
Adding the above two identities and then using \eqref{R-GH} and  \eqref{GHf5}, we obtain
\begin{align*}
\Lambda_3&=2\frac{f_1f_2 f_{10}}{f_5^3}\left(G^3(q^2)H(q)+qG(q)H^3(q^2)\right),
\end{align*}
which by \eqref{newGH1} becomes
\begin{align}
\label{Ident4}\Lambda_3&=2\frac{f_2^2 f_{10}^{10}}{f_4 f_5^8 f_{20}^3}+8 q^2\frac{f_1 f_4 f_{10} f_{20}^3}{f_2 f_5^5}=2\frac{f_{10}^{10}}{f_5^8 f_{20}^3}\varphi(-q^2)+8 q^2\frac{f_{10} f_{20}^3}{f_5^5}\psi(-q).
\end{align}
Employing \eqref{5dissectPhi} and \eqref{5dissectPsi}, we rewrite \eqref{Ident4} as
\begin{align*}
\Lambda_3&=2\frac{f_{10}^{10}}{f_5^8 f_{20}^3}\left(\varphi(-q^{50})-2q^2f(-q^{30},-q^{70})+2q^8f(-q^{10},-q^{90})\right)\notag\\
&\quad+8 q^2\frac{f_{10}f_{20}^3}{f_5^5}\left(f(q^{10},-q^{15})-qf(-q^{5},q^{20})-q^3\psi(-q^{25})\right).
\end{align*}
The above identity gives $U_{5n+r}(\Lambda_3)=0$ for $r\in\{1,4\}$, which is equivalent to (iii).

\noindent\textit{Proof of (iv)} We set
\begin{align*}
\Lambda_4:=\sum_{n=0}^{\infty}\lambda_{4}(n)q^n+q^2\sum_{n=0}^{\infty}\lambda_{4}^{\prime}(n)q^n
&=\dfrac{1}{R(q)R^2(q^2)}+q^2R(q)R^2(q^2).
\end{align*}

Adding \eqref{RobinsKey1} and \eqref{RobinsKey2} and then rearranging, we have
\begin{align}
\label{RobinsKey3} 1+q\dfrac{H\left(q\right)H^2\left(q^2\right)}{G\left(q\right)G^2\left(q^2\right)}&=\dfrac{G\left(q\right)H\left(q^2\right)}{G^2\left(q^2\right)}\dfrac{f_{5}^2}{f_{10}^2}.
\end{align}
Subtracting \eqref{RobinsKey1} from \eqref{RobinsKey2}, and then rearranging, we have
\begin{align}
\label{RobinsKey4}\dfrac{G\left(q\right)G^2\left(q^2\right)}{H\left(q\right)H^2\left(q^2\right)}-q&=\dfrac{G\left(q^2\right)H\left(q\right)}{H^2\left(q^2\right)}\dfrac{f_{5}^2}{f_{10}^2}.
\end{align}
With the aid of \eqref{RobinsKey3}, \eqref{RobinsKey4}, \eqref{R-GH}, and \eqref{GHf5}, we obtain
\begin{align*}
\Lambda_4&=\dfrac{f_5^2f_2^2}{f_{10}^4}\left(G^3\left(q^2\right)H\left(q\right)+qG\left(q\right)H^3\left(q^2\right)\right).
\end{align*}
Employing \eqref{newGH1} and \eqref{fq}, we find that
\begin{align*}
\Lambda_4&=\frac{f_2^3 f_{10}^5}{f_1 f_4 f_5^3 f_{20}^3}+4 q^2\frac{f_4 f_{20}^3}{f_{10}^4}=f(q)\frac{f_{10}^5}{f_5^3 f_{20}^3}+4 q^2f_4 \frac{f_{20}^3}{f_{10}^4}.
\end{align*}
Employing \eqref{5df_1} in the above identity, we have
\begin{align*}
\Lambda_4&=\dfrac{f_{10}^5f(q^{25})}{f_5^3 f_{20}^3}\left(\dfrac{1}{R\left(-q^5\right)}+q-q^2R\left(-q^5\right)\right)\\
&\quad+4q^2\dfrac{f_{20}^3f_{100}}{f_{10}^4}\left(\dfrac{1}{R\left(q^{20}\right)}-q^4-q^8R\left(q^{20}\right)\right),
\end{align*}
from which we arrive at $U_{5n+r}\left(\Lambda_4\right)=0$ for $r\in\{3,4\}$. This finishes the proof of (iv).

\noindent\textit{Proof of (v)}
Set
\begin{align*}
\Lambda_5&:=\sum_{n=0}^{\infty}\lambda_{5}(n)q^n-q^2\sum_{n=0}^{\infty}\lambda_{5}^{\prime}(n)q^n=\dfrac{1}{R(q)R(q^4)}-q^2R(q)R(q^4).
\end{align*}

Now, recall from Gugg \cite[Theorem 3.3 (i) and (ii)]{G2} that
\begin{align}
\label{GuggMu1}\dfrac{1+qR\left(q\right)R\left(q^{4}\right)}{1-qR\left(q\right)R\left(q^{4}\right)}&=\dfrac{\varphi(q)}{\varphi\left(q^5\right)}=\dfrac{f_2^5f_5^2f_{20}^2}{f_1^2f_4^2f_{10}^5},\\
\label{GuggMu2}\dfrac{R\left(q\right)R\left(q^{4}\right)}{(1-qR\left(q\right)R\left(q^{4}\right))^2}&=\dfrac{(-q;q^2)_\infty}{(-q^5;q^{10})_\infty^5}
=\dfrac{f_2^2f_5^5f_{20}^5}{f_1f_4f_{10}^{10}}.
\end{align}
Multiplying the numerator and the denominator of the left side of \eqref{GuggMu1} by the denominator, and then using \eqref{GuggMu2} and \eqref{5df_1}, we find that
\begin{align}\label{RqRq4}
\Lambda_5=\dfrac{f_2^3f_{10}^5}{f_1f_4f_5^3f_{20}^3}=\dfrac{f_{10}^5}{f_5^3f_{20}^3}f(q)=\dfrac{f_{10}^5f(q^{25})}{f_5^3f_{20}^3}\left(\dfrac{1}{R(-q^5)}+q-q^2R(-q^5)\right).
\end{align}
It follows  from \eqref{RqRq4} that $U_{5n+r}(\Lambda_5)=0$ for $r\in\{3,4\}$, which is equivalent to (v).

\noindent\textit{Proof of (vi)} We set
\begin{align*}
\Lambda_6&:=\sum_{n=0}^{\infty}\lambda_{6}(n)q^n-q^2\sum_{n=0}^{\infty}\lambda_{6}^{\prime}(n)q^n=\dfrac{1}{R(-q)R(-q^4)}-q^2R(-q)R(-q^4).
\end{align*}

Employing \eqref{R-GH}, \eqref{Rama2}, \eqref{newGH3}, and \eqref{GHf5}, we find that
\begin{align*}
\Lambda_6&=\dfrac{G\left(-q\right)G\left(-q^4\right)}{H\left(-q\right)H\left(-q^4\right)}-q^2\dfrac{H\left(-q\right)H\left(-q^4\right)}{G\left(-q\right)G\left(-q^4\right)}\\
&=\dfrac{\left(G(-q)G(-q^4)\right)^2-q^2\left(H(-q)H(-q^4)\right)^2}{G(-q)H(-q)G(-q^4)H(-q^4)}\\
&=\left(\frac{f_2f_8^2 f_{20}^4 f_{80}}{f_{10}^3 f_{16} f_{40}^4}-2q\frac{f_2 f_4 f_{20}}{f_{10}^3}\right)\cdot\dfrac{f_5}{f_1}.
\end{align*}
Employing the 2-dissection of $f_5/f_1$ from Lemma \ref{Lemma-2-dissection} and then applying $U_{2n+1}$, we obtain
\begin{align}
\label{E_6Trick1}U_{2n+1}\left(\Lambda_6\right)=\frac{f_2^3 f_4 f_{10}^3 f_{40}}{f_1^2 f_5^2 f_8 f_{20}^3}-2\frac{f_2 f_4 f_{10}^3}{f_1 f_5^3 f_{20}}=\frac{f_2 f_4 f_{10}^3 f_{40}}{f_1 f_5^3 f_8 f_{20}^3}\left(\frac{f_2^2 f_5}{f_1}-2\frac{f_8 f_{20}^2}{f_{40}}\right).
\end{align}

Now, putting $a=q,b=q^9,c=q^3,$ and $d=q^7$ in \eqref{ThetaIdentity2}, we have
\begin{align}
\label{E_6Trick2}f\left(q,q^9\right)f\left(q^3,q^7\right)+f\left(-q,-q^9\right)f\left(-q^3,-q^7\right)=2f\left(q^4,q^{16}\right)f\left(q^8,q^{12}\right).
\end{align}

Using Jacobi triple product identity \eqref{jtpi}, we have
\begin{align}\label{E_6Trick3}
 f\left(-q,-q^9\right)f\left(-q^3,-q^7\right)&=\dfrac{f_1f_{10}^3}{f_2f_5}, &
f\left(q,q^4\right)f\left(q^2,q^3\right)&=\dfrac{f_2f_5^3}{f_1f_{10}}.
\end{align}
Therefore, \eqref{E_6Trick2} can be rewritten in the form
\begin{align}
\label{Sameway1}\frac{f_2^2 f_5}{f_1}-2\frac{f_8 f_{20}^2}{f_{40}}=-\dfrac{f_1f_4f_{20}^3}{f_2f_5f_{20}}.
\end{align}
Employing \eqref{5dissectPhi} and  \eqref{Sameway1}, \eqref{E_6Trick1} can be written as
\begin{align*}
U_{2n+1}\left(\Lambda_6\right)&=-\dfrac{f_4^2f_{10}^6f_{40}}{f_8f_5^4f_{20}^4}=-\dfrac{f_{10}^6f_{40}}{f_5^4f_{20}^4}\varphi\left(-q^4\right)\\
&=-\dfrac{f_{10}^6f_{40}}{f_5^4f_{20}^4}\left(\varphi(-q^{100})-2q^4f(-q^{60},-q^{140})+2q^{16}f(-q^{20},-q^{180})\right).
\end{align*}
The above identity gives $U_{10n+r}(\Lambda_6)=0$ for $r\in\{5,7\}$, which is equivalent to (vi).

\noindent\textit{Proofs of (vii) and (viii)}
Let
\begin{align*}
\Lambda_7&:=\sum_{n=0}^{\infty}\lambda_{7}(n)q^n+\sum_{n=0}^{\infty}\lambda_{7}^{\prime}(n)q^n=\dfrac{R^4(q)}{R\left(q^{4}\right)}+\dfrac{R\left(q^{4}\right)}{R^4(q)},\\
\Lambda_8&:=\sum_{n=0}^{\infty}\lambda_{8}(n)q^n+q^4 \sum_{n=0}^{\infty}\lambda_{8}^{\prime}(n)q^n=\dfrac{R^2(q)}{R^3\left(q^{4}\right)}+q^4\dfrac{R^3\left(q^{4}\right)}{R^2(q)}.
\end{align*}
To prove (vii) and (viii), we show that $U_{10n+r}(\Lambda_7)=0$ for $r\in\{1,9\}$ and $U_{10n+r}(\Lambda_8)=0$ for $r\in\{5,9\}$.
We note that
\begin{align}\label{Lam7}
\Lambda_7&=\dfrac{R(q^4)}{R^2(q^2)}+\dfrac{R^2(q^2)}{R(q^4)}
-\left(\dfrac{R(q^2)}{R^2(q)}-\dfrac{R^2(q)}{R(q^2)}\right)\left(\dfrac{R^2(q)R(q^2)}{R(q^4)}-\dfrac{R(q^4)}
{R^2(q)R(q^2)}\right)\\
\label{Lam8}
\Lambda_8&=\left(\dfrac{R^2\left(q\right)R\left(q^2\right)}{R\left(q^4\right)}-\dfrac{R\left(q^4\right)}{R^2\left(q\right)R\left(q^2\right)}\right)\left(\dfrac{1}{R(q^2)R^2\left(q^4\right)}
-q^4R(q^2)R^2\left(q^4\right)\right)\notag\\
&\quad+\left(\dfrac{1}{R(q)R^2\left(q^2\right)}-q^2R(q)R^2\left(q^2\right)\right)\left(\dfrac{1}{R(q)R\left(q^4\right)}-q^2R(q)R\left(q^4\right)\right)\notag\\
&\quad+q^2\left(\dfrac{R^2\left(q\right)}{R\left(q^2\right)}+\dfrac{R\left(q^2\right)}{R^2\left(q\right)}\right).
\end{align}

From Baruah and Begum \cite[Lemma 1.3]{BBegum1}, we recall that
\begin{align}
\label{Ident3}\dfrac{1}{R(q)R^2(q^2)}-q^2R(q)R^2(q^2)&=\dfrac{f_2f_{5}^5}{f_1f_{10}^5},\\
\label{Ident8}\dfrac{R(q^2)}{R^2(q)}-\dfrac{R^2(q)}{R(q^2)}&=4q\dfrac{f_1f_{10}^5}{f_2f_{5}^5}.
\end{align}

Gugg \cite[Theorem 3.6 (i) and (ii)]{G2} proved that
\begin{align}\label{GuggMu11}
\dfrac{1+R^2(q)R(q^2)/R(q^4)}{1-R^2(q)R(q^2)/R(q^4)}&=\dfrac{\psi(q^2)}{q\psi(q^{10})}=\dfrac{f_4^2f_{10}}{qf_2f_{20}^2},\\
\label{GuggMu22}\dfrac{4R^2(q)R(q^2)/R(q^4)}{\left(1-R^2(q)R(q^2)/R(q^4)\right)^2}&=\dfrac{f_4f_{10}^5}{q^2f_2f_{20}^5}.
\end{align}
Multiplying the numerator and the denominator of the left side of \eqref{GuggMu11} by the denominator, and then using \eqref{GuggMu22}, we find that
\begin{align}
\label{Ident2}\dfrac{R^2(q)R(q^2)}{R(q^4)}-\dfrac{R(q^4)}{R^2(q)R(q^2)}&=-4q\dfrac{f_4f_{20}^3}{f_{10}^4}.
\end{align}

Employing \eqref{Ident4} with $q$ replaced by $q^2$, \eqref{Ident2}, and \eqref{Ident8} in \eqref{Lam7}, we obtain
\begin{align*}
\Lambda_7&=2\frac{f_4^2 f_{20}^{10}}{f_8 f_{10}^8 f_{40}^3}+16q^2\frac{f_4 f_{10}f_{20}^3 }{f_2}\cdot\dfrac{f_1}{f_5}\cdot\dfrac{1}{f_5^4}+8q^4\frac{ f_2f_8 f_{20}f_{40}^{3} }{f_4 f_{10}^5}.
\end{align*}
Using the 2-dissections of $f_1/f_5$ and $1/f_1^4$ with $q$ replaced by $q^5$ from Lemma \ref{Lemma-2-dissection} in the above identity, and then applying the operator $U_{2n+1}$, we obtain
\begin{align*}
U_{2n+1}\left(\Lambda_7\right)=
-16q\frac{f_{10}^{17} }{f_{5}^{15} f_{20}^3}f(q)+64q^3\frac{f_{10}^8 f_{20}^3 }{f_{5}^{12}}f_4.
\end{align*}
With the help of \eqref{5df_1}, the above identity becomes
\begin{align}
\label{NewJune8} U_{2n+1}\left(\Lambda_7\right)=&
-16q\frac{f_{10}^{17}f(q^{25}) }{f_{5}^{15} f_{20}^3}\left(\dfrac{1}{R(-q^5)}+q-q^2R(-q^5)\right)\notag\\
&+64q^3\dfrac{f_{10}^8 f_{20}^3 f_{100}}{f_{5}^{12}}f_{25}\left(\dfrac{1}{R(q^{20})}-q^4-q^8R(q^{20})\right).
\end{align}
Applying $U_{5n+r}$ for $r\in\{0,4\}$ in \eqref{NewJune8}, we have $U_{10n+r}(\Lambda_7)=0$ for $r\in\{1,9\}$.

Again, employing \eqref{Ident4}, \eqref{RqRq4}, \eqref{Ident2}, and \eqref{Ident3} with $q$ replaced by $q^2$ in \eqref{Lam8}, we find that
\begin{align*}
\Lambda_8&=\frac{f_2^4}{f_4 f_{20}^3}\cdot\dfrac{f_5^2}{f_1^2}-4q\frac{f_4^2 f_{10} }{f_2 f_{20}^2}+2q^2\frac{ f_4^2 f_{20}^{10} }{f_8 f_{10}^8 f_{40}^3}+8q^6\frac{ f_2 f_8 f_{20} f_{40}^3 }{f_4 f_{10}^5}.
\end{align*}
Using the 2-dissection of $f_5/f_1$ from Lemma \ref{Lemma-2-dissection} in the above identity and then applying $U_{2n+1}$ and \eqref{5dissectPsi}, we obtain
\begin{align*}
U_{2n+1}\left(\Lambda_8\right)&=-2\dfrac{f_2^2f_5}{f_1f_{10}}=-2\dfrac{f_5}{f_{10}}\psi(q)\notag\\
&=-2\dfrac{f_5}{f_{10}}\left(f(q^{10},q^{15})+qf(q^{5},q^{20})+q^3\psi(q^{25}\right),
\end{align*}
from which it follows that $U_{10n+r}(\Lambda_8)=0$ for $r\in\{5,9\}$. Thus, we complete the proofs of (vii) and (viii).

\noindent\textit{Proof of (ix)}
Set
\begin{align*}
\Lambda_9&:=\sum_{n=0}^{\infty}\lambda_{9}(n)q^n-q^6 \sum_{n=0}^{\infty}\lambda_{9}^{\prime}(n)q^n=\dfrac{R(q)}{R\left(q^{16}\right)}-q^6\dfrac{R\left(q^{16}\right)}{R(q)}.
\end{align*}
Due to \eqref{R-GH}, \eqref{GHf5}, \eqref{Rama3}, and \eqref{newGH2}, we find that
\begin{align*}
\Lambda_9=\dfrac{G\left(q^{16}\right)H\left(q\right)}{G\left(q\right)H\left(q^{16}\right)}-q^6 \dfrac{G\left(q\right)H\left(q^{16}\right)}{G\left(q^{16}\right)H\left(q\right)}
&=\dfrac{f_1f_{16}}{f_5f_{80}}\left(G^2(q^{16})H^2(q)-q^6G^2(q)H^2(q^{16})\right)\\
&=\left(\dfrac{f_4^2 f_{16} f_{20}^2}{f_2^2f_8 f_{40} f_{80}}+2q^3\dfrac{f_4^2 }{f_2^2}\right)\cdot \dfrac{f_1}{f_5}.
\end{align*}
Employing the 2-dissection of $f_1/f_5$ from Lemma \ref{Lemma-2-dissection} and then extracting the terms involving the odd powers of $q$, we find that
\begin{align}
\label{N5Trick1}U_{2n+1}\left(\Lambda_9\right)=-\dfrac{f_2^2f_8f_{10}^2}{f_1f_4f_5^3f_{20}f_{40}}\left(\dfrac{f_2^2f_5f_{20}}{f_1f_4}-2q\dfrac{f_4^2f_{10}f_{40}}{f_2f_8}\right).
\end{align}

Now, putting $a=q,b=q^9,c=q^3,$ and $d=q^7$ in \eqref{ThetaIdentity3}, we have
\begin{align*}
f(q,q^9)f(q^3,q^7)-f(-q,-q^9)f(-q^3,-q^7)=2q f(q^2,q^{18})f(q^6,q^{14}),
\end{align*}
which, with the help of \eqref{E_6Trick3} and \eqref{fq}, can be rewritten as
\begin{align}\label{ixnew}
\dfrac{f_2^2f_5f_{20}}{f_1f_4}-2q\dfrac{f_4^2f_{10}f_{40}}{f_2f_8}=\dfrac{f_1f_{10}^3}{f_2f_5}.
\end{align}
Employing \eqref{ixnew} in \eqref{N5Trick1} and then using \eqref{5dissectPsi}, we have
\begin{align*}
U_{2n+1}\left(A_9\right)&=-\dfrac{f_2f_8f_{10}^5}{f_4f_5^4f_{20}f_{40}}=-\dfrac{f_{10}^5}{f_5^4f_{20}f_{40}}\psi(-q^2)\\
&=-\dfrac{f_{10}^5}{f_5^4f_{20}f_{40}}\left(f(q^{20},-q^{30})-q^2f(-q^{10},q^{40})-q^6\psi(-q^{50})\right).
\end{align*}
The above identity ensures that $U_{10n+r}(\Lambda_9)=0$ for $r\in\{7,9\}$, which is equivalent to (ix).

\noindent\textit{Proofs of (x) and (xi)} From \eqref{GuggMu22}, we have
 \begin{align*}
\sum_{n=0}^\infty\lambda_{10}(n)q^n+\sum_{n=0}^\infty\lambda_{10}^\prime(n)q^n=\dfrac{R^2\left(q\right)R\left(q^2\right)}{R\left(q^4\right)}+\dfrac{R\left(q^4\right)}{R^2\left(q\right)R\left(q^2\right)}&=2+4q^2\dfrac{f_2f_{20}^5}{f_4f_{10}^5}.
\end{align*}
Due to \eqref{Ident2} and the above identity, it readily follows that
\begin{align*}
\lambda_{10}(2n+1)=-\lambda_{10}^\prime(2n+1),\quad
\lambda_{10}(2n)=\lambda_{10}^\prime(2n),
\end{align*}
which together imply (x).

Next, with the aid of \eqref{5df_1}, we can rewrite \eqref{Ident2} as
\begin{align*}
\sum_{n=0}^\infty\lambda_{10}(n)q^n-\sum_{n=0}^\infty\lambda_{10}^\prime(n)q^n&=-4q\dfrac{f_{20}^3f_{100}}{f_{10}^4}\left(\dfrac{1}{R(q^{20})}-q^4-q^8R(q^{20})\right).
\end{align*}
Equating the coefficients of $q^{5n+2}$ and $q^{5n+3}$ from both sides of the identity above, we arrive at $\lambda_{10}(5n+r)=\lambda_{10}^\prime(5n+r)$ for $r\in\{2,3\}$. Thus, we
complete the proof of (xi) as well as Theorem \ref{Theo5_1}.\qed

\section{Proof of Theorem \ref{Theo8_1}}\label{sec6}

Proofs of the identities in Theorem \ref{Theo8_1} are similar in nature. Therefore, we prove only (xv)--(xix). We set
\begin{align*}
P_1&:=\sum_{n=0}^{\infty}\rho_{5,1}(n)q^n+q^{4}\sum_{n=0}^{\infty}\rho_{5,1}^{\prime}(n)q^n\\
&=\dfrac{S\left(-q\right)T\left(q\right)S\left(q^4\right)}{S\left(q\right)T\left(-q\right)T\left(q^4\right)}+q^4\dfrac{S\left(q\right)T\left(-q\right)T\left(q^4\right)}{S\left(-q\right)T\left(q\right)S\left(q^4\right)},\\
P_2&:=\sum_{n=0}^{\infty}\rho_{5,1}(n)q^n-q^{4}\sum_{n=0}^{\infty}\rho_{5,1}^{\prime}(n)q^n\\
&=\dfrac{S\left(-q\right)T\left(q\right)S\left(q^4\right)}{S\left(q\right)T\left(-q\right)T\left(q^4\right)}-q^4\dfrac{S\left(q\right)T\left(-q\right)T\left(q^4\right)}{S\left(-q\right)T\left(q\right)S\left(q^4\right)},\\
P_3&:=\sum_{n=0}^{\infty}\rho_{5,2}(n)q^n+q^{8}\sum_{n=0}^{\infty}\rho_{5,2}^{\prime}(n)q^n\\
&=\dfrac{S^2\left(-q\right)T^2\left(q\right)S^2\left(q^4\right)}{S^2\left(q\right)T^2\left(-q\right)T^2\left(q^4\right)}+q^8\dfrac{S^2\left(q\right)T^2\left(-q\right)T^2\left(q^4\right)}{S^2\left(-q\right)T^2\left(q\right)S^2\left(q^4\right)}
\end{align*}

Using \eqref{ModularIdentityST1}, \eqref{STf}, and the first identity of Lemma \ref{Lemma-2-dissection}, we have
\begin{align*}
P_1&=\dfrac{\left(S\left(-q\right)T\left(q\right)S\left(q^4\right)-q^2  S\left(q\right)T\left(-q\right)T\left(q^4\right)\right)^2}{S\left(q\right)T\left(q\right)S\left(-q\right)T\left(-q\right)S\left(q^4\right)T\left(q^4\right)}+2q^2
=2q^2+\dfrac{f_4^2 f_{16}^8}{f_2 f_8^7 f_{32}^4}\cdot f_1^2\\
&=2q^2+\dfrac{f_{16}^6}{f_8^2 f_{32}^4}- 2q\dfrac{f_4^2 f_{16}^{10}}{f_8^8}.
\end{align*}
Therefore, it follows that $U_{4n+6}(P_1)=0$ and $U_{8n+4}(P_1)=0$, which are equivalent to (xv) and (xvii), respectively.

Next, using \eqref{ModularIdentityST1}, \eqref{ModularIdentityST2}, \eqref{STf}, and the first identity of Lemma \ref{Lemma-2-dissection}, we find that
\begin{align}\label{P2}
P_2&=\dfrac{\left(S\left(-q\right)T\left(q\right)S\left(q^4\right)\right)^2-q^4  \left(S\left(q\right)T\left(-q\right)T\left(q^4\right)\right)}
{S(q)T(q)S(-q)T(-q)S(q^4)T(q^4)}\notag\\
&=\left(\dfrac{f_4 f_{16}^6}{f_2^3 f_8^2 f_{32}^4}+4q^3\dfrac{f_4^3 f_{16}^4 }{f_2^3 f_8^6}\right)\cdot f_1^2\notag\\
&=\left(\dfrac{f_4 f_{16}^6}{f_2^3 f_8^2 f_{32}^4}+4q^3\dfrac{f_4^3 f_{16}^4 }{f_2^3 f_8^6}\right)\left(\frac{f_{2} f_{8}^5}{f_{4}^2 f_{16}^2}- 2q\frac{ f_{2} f_{16}^2}{f_{8}}\right).
\end{align}
Therefore,
\begin{align*}
U_{2n+1}(P_2)=\left(-\frac{2 f_2 f_8^8}{f_4^3 f_{16}^4}+4q\frac{ f_2 f_8^2 }{f_4}\right) \dfrac{1}{f_1^2}.
\end{align*}
Employing the $2$-dissection of $1/f_1^2$ from Lemma \ref{Lemma-2-dissection} in the above identity, we find that
\begin{align}
\label{GoOne}U_{2n+1}(P_2)=-2\frac{ f_8^{13}}{f_2^4 f_4^3 f_{16}^6}+8q^2\frac{f_4 f_8 f_{16}^2 }{f_2^4},
\end{align}
 which gives $U_{4n+7}(P_2)=0$, which is equivalent to (xvi).

From \eqref{GoOne}, we also have
\begin{align*}
U_{4n+1}(P_2)=-2\frac{f_4^{13}}{f_1^4 f_2^3 f_{8}^6}+8q\frac{f_2 f_4 f_{8}^2 }{f_1^4},
\end{align*}
which can be rewritten using the $2$-dissection of $1/f_1^4$ in Lemma \ref{Lemma-2-dissection}  as
\begin{align*}
U_{4n+1}(P_2)=-2\frac{f_4^{27}}{f_2^{17} f_8^{10}}+32q^2\frac{f_4^3 f_8^6 }{f_2^9}.
\end{align*}
From the above identity, we have $U_{8n+5}(P_2)=0$, which is equivalent to (xviii).

Now, we prove (xix). Using the expression for $P_2$ from \eqref{P2}, we have
\begin{align*}
P_3=P_2^2+2q^4=\dfrac{f_1^4 f_4^2 f_{16}^{12}}{f_2^6 f_8^4 f_{32}^8}+16q^6\dfrac{ f_1^4 f_4^6 f_{16}^8 }{f_2^6 f_8^{12}}+8q^3\dfrac{ f_1^4 f_4^4 f_{16}^{10} }{f_2^6 f_8^8 f_{32}^4}+2 q^4.
\end{align*}
Applying the 2-dissection of $f_1^4$ from Lemma \ref{Lemma-2-dissection} in the identity above, and then applying the extraction operator $U_{2n}$, we find that
\begin{align*}
U_{2n}(P_3)=\dfrac{f_8^{12} f_2^{12}}{f_1^8 f_4^8 f_{16}^8}+2 q^2-32q^2\dfrac{f_8^{10} f_2^2 }{f_1^4 f_4^4 f_{16}^4}+16q^3\dfrac{f_8^8 f_2^{16} }{f_1^8 f_4^{16}}.
\end{align*}
Again, invoking the 2-dissection of $1/f_1^4$ from Lemma \ref{Lemma-2-dissection} in the identity above, and then extracting the terms involving $q^{2n}$, we obtain
\begin{align*}
U_{4n}(P_3)=\frac{f_4^4 f_2^{20}}{f_1^{16} f_8^8}+2 q+16q\frac{f_4^{20} }{f_1^8 f_2^4 f_8^8}-32q\frac{ f_4^6 f_2^{10} }{f_1^{12} f_8^4}+128 q^2\frac{ f_4^8}{f_1^8},
\end{align*}
which, with the help of the 2-dissection of $1/f_1^4$ from Lemma \ref{Lemma-2-dissection}, can be rewritten as
\begin{align*}
U_{4n}(P_3)=\frac{f_4^{60}}{f_2^{36} f_8^{24}}+2 q-32q^2\frac{f_4^{36} }{f_2^{28} f_8^8}+256q^4\frac{f_8^8 f_4^{12} }{f_2^{20}}.
\end{align*}
Therefore, $U_{8n+12}(P_3)=0$, which is equivalent to (xix).

The remaining results in Theorem \ref{Theo8_1} can be proved in a similar fashion by using some modular relations between $S(q)$ and $T(q)$ found by Huang \cite{H1} and Xia and Yao \cite{XY1} that are analogous to \eqref{ModularIdentityST1} and \eqref{ModularIdentityST2}. In the following table, we mention the locations of the corresponding relations used to prove the results.

\begin{align*}
\begin{tabular}{c c}
 \hline
  Results in Theorem \ref{Theo8_1}  & Used identities analogous to  \eqref{ModularIdentityST1} and \eqref{ModularIdentityST2}\T\B\\  \hline
    (i)--(vi) & \cite[(2.1), (2.2)]{H1}\T\B\\
     (vii)--(x) & \cite[(2.5), (2.7), (2.3),(2.9)]{H1} \T\B\\
 (xi)--(xiii) & \cite[(2.4), (2.6), (2.10)]{H1} \T\B\\
  (xiv) & \cite[(2.8)]{XY1}\T\B\\
  (xx) & \cite[(2.9)]{XY1} \T\B\\
   (xxi) & \cite[(2.10)]{XY1} \T\B\\
    (xxii) & \cite[(2.11)]{XY1} \T\B\\
     (xxiii) & \cite[(2.12)]{XY1} \T\B\\\hline
\end{tabular}
\end{align*}\qed

\section{Proof of Theorem \ref{Theo12_1}}\label{sec7}
 Since the proofs of the results in Theorem \ref{Theo12_1} are similar in nature, we choose to prove (xv) and (xvi) only,  which are somewhat trickier to prove than the others.
We set
\begin{align}
\label{Lambda32}\Xi_{3,2}&:=\sum_{n=0}^{\infty}\xi_{3,2}(n)q^n-q^8\sum_{n=0}^{\infty}\xi_{3,2}^{\prime}(n)q^n
=\dfrac{X^2\left(q\right)Y^2\left(q^3\right)}{X^2\left(q^3\right)Y^2\left(q\right)}-q^8\dfrac{X^2\left(q^3\right)Y^2\left(q\right)}{X^2\left(q\right)Y^2\left(q^3\right)}.
\end{align}
To prove (xv) and (xvi), it is enough to show that $U_{24n+14}\left(\Xi_{3,2}\right)=0$ and \linebreak $U_{36n+25}\left(\Xi_{3,2}\right)=0$, respectively.

We have
\begin{align*}
\Xi_{3,2}&=\dfrac{\left(X^2\left(q\right)Y^2\left(q^3\right)-q^4X^2\left(q^3\right)Y^2\left(q\right)\right)}{X^2\left(q\right)Y^2
\left(q\right)X^2\left(q^3\right)Y^2\left(q^3\right)}\\
&\quad\times\Big(\left(X\left(q\right)Y\left(q^3\right)\mp q^2X\left(q^3\right)Y\left(q\right)\right)^2 \pm 2q^2X\left(q\right)Y\left(q\right)X\left(q^3\right)Y\left(q^3\right)\Big).
\end{align*}

Using \eqref{Xq}, \eqref{Yq}, and \eqref{jtpi}, it is easy to show that
\begin{align}\label{XYq}X(q)Y(q)=\dfrac{f_{12}^2f_6}{f_1f_2f_3}.
\end{align}

Employing \eqref{BBXY1}, \eqref{RobinsXY1}, and \eqref{XYq} in \eqref{Lambda32}, we find that
\begin{align}
\label{First12Gen}\Xi_{3,2}&=\frac{ f_4 f_6^5 f_{18}^3}{ f_{12}^6 f_{36}^3}\cdot\frac{f_1^2}{f_3^2}+2q^2\frac{ f_4 f_6^5 f_9}{f_2 f_3^2 f_{12}^4 f_{36}}\cdot\frac{f_1}{f_3}\cdot\frac{f_9}{f_3},\\
\label{Second12Gen}\Xi_{3,2}&=\frac{ f_6^{15} f_9^4}{ f_3^6 f_{12}^{10} f_{18}^3 f_{36}}\cdot\dfrac{f_1^2 f_4^2}{f_2^2}\cdot \dfrac{f_4}{f_2^2}-2q^2\frac{f_6^5 f_9}{ f_3^2 f_{12}^4 f_{36}}\cdot\dfrac{f_1 f_4}{f_2}.
\end{align}

For convenience, we use \eqref{First12Gen} and \eqref{Second12Gen} to prove (xv) and (xvi), respectively. First we prove (xv).

Invoking the 2-dissections of $f_1^2/f_3^2$, $f_1/f_3$, and $f_3/f_1$ from Lemma \ref{Lemma-2-dissection} in \eqref{First12Gen} and then extracting the terms involving the even powers of $q$, we find that
\begin{align*}
U_{2n}(\Xi_{3,2})&=\frac{f_1 f_2^3 f_9^3}{f_4 f_6^2 f_{12} f_{18}^3}+2q\frac{f_2 f_3 f_8 f_9 f_{12} f_{36}^2}{f_4 f_6^3 f_{18}^2 f_{72}}-2q^3\frac{ f_3 f_4^2 f_9 f_{12} f_{72} }{f_6^3 f_8 f_{18} f_{36}}.
\end{align*}
 From the above identity, employing the suitable $2$- and $3$-dissections from Lemmas \ref{Lemma-2-dissection} and \ref{Lemma-3-dissection}, we successively find  $U_{6n+2}(\Xi_{3,2})$ and $U_{12n+2}(\Xi_{3,2})$. Extracting the terms involving the odd powers of $q$ in the expression for  $U_{12n+2}(\Xi_{3,2})$, we obtain
\begin{align}
\label{U24n+14}U_{24n+14}(\Xi_{3,2})&=-6\frac{f_2^4 f_3^3}{f_1^7}+4\frac{f_4^4 f_6^{11}}{f_1^4 f_2 f_3^6 f_{12}^4}+2\frac{f_2^8 f_6^8}{f_1^7 f_3^5 f_4^2 f_{12}^2}+8q\frac{ f_2^2 f_4^2 f_6^2 f_{12}^2}{f_1^5 f_3^3}\notag\\
&\quad+4q\frac{f_2^{11} f_{12}^4}{f_1^8 f_3^2 f_4^4 f_6}\notag\\
&=-6\frac{f_2^4 f_3^3}{f_1^7}+\left(4\frac{f_2^2f_4^3 f_6^{2}}{f_1^5  f_3^3 f_{12}}+2\frac{f_2^{11} f_{12}}{f_1^8 f_3^2 f_4^3 f_{6}}\right)\left(\dfrac{f_1f_4f_{6}^9}{f_{2}^3f_3^3f_{12}^3}+2q\dfrac{f_{12}^3}{f_4}\right).
\end{align}

Now,  the eighth identity of Lemma \ref{Lemma-2-dissection} is
\begin{align}\label{f33}
\frac{f_{3}^3}{f_1}=\frac{f_{4 }^3 f_{6 }^2}{f_{2 }^2 f_{12 }}+q\frac{f_{12 }^3 }{f_{4}}.
\end{align}
Replacing $q$ by $-q$ in \eqref{f33} and then subtracting the resulting identity from \eqref{f33}, it follows that
\begin{align*}
\dfrac{f_3^3}{f_1}-\dfrac{f^3(q^3)}{f(q)}=2q\dfrac{f_{12}^3}{f_4},
\end{align*}
which, by \eqref{fq}, reduces to
\begin{align}
\label{TrickIden1}\dfrac{f_3^3}{f_1}=\dfrac{f_1f_4f_{6}^9}{f_{2}^3f_3^3f_{12}^3}+2q\dfrac{f_{12}^3}{f_4}.
\end{align}
Employing \eqref{TrickIden1} in \eqref{U24n+14}, we have
\begin{align}
\label{U24n+14a}U_{24n+14}(\Xi_{3,2})&=-6\frac{f_2^4 f_3^3}{f_1^7}+4\frac{f_2^2f_4^3 f_6^{2}}{f_1^6   f_{12}}+2\frac{f_2^{11} f_{12}}{f_1^6  f_4^3 f_{6}}\cdot\dfrac{f_3}{f_1^{3}}.
\end{align}
We now apply the 2-dissection of $f_3/f_1^3$ from Lemma \ref{Lemma-2-dissection} in \eqref{U24n+14a}, and then use \eqref{f33} to obtain
\begin{align*}
U_{24n+14}(\Xi_{3,2})&=-6\frac{f_2^4 f_3^3}{f_1^7}+4\frac{f_2^2f_4^3 f_6^{2}}{f_1^6   f_{12}}+2\frac{f_2^{11} f_{12}}{f_1^6  f_4^3 f_{6}}\left(\dfrac{f_4^6f_6^3}{f_2^9f_{12}^2}+3q\dfrac{f_4^2f_6f_{12}^2}{f_2^7}\right)\\
&=-6\frac{f_2^4 f_3^3}{f_1^7}+6\dfrac{f_2^4}{f_1^6}\left(\dfrac{f_4^3f_6^2}{f_2^2f_{12}}+q\dfrac{f_{12}^3}{f_4}\right),\\
&=-6\frac{f_2^4 f_3^3}{f_1^7}+6\frac{f_2^4 f_3^3}{f_1^7}=0.
\end{align*}
This completes the proof of (xv).

Now, we prove (xvi). Employing the 3-dissections of $f_1 f_4/f_2$ and $f_2/f_1^2$ from Lemma \ref{Lemma-3-dissection} in \eqref{Second12Gen} and then extracting the terms involving $q^{3n+1}$ and finally using \eqref{TrickIden1}, we have
\begin{align}
\label{36n+25Trick1}U_{3 n+1}(\Xi_{3,2})&=-2\frac{f_2^5 f_3^3 f_6^7}{f_1^5 f_4^5 f_{12}^5}+4q\frac{f_6^7 f_2^5}{f_1^4 f_4^6 f_{12}^2}+2q\frac{f_3^6 f_{12} f_2^8}{f_1^6 f_4^7 f_6^2}\notag\\
&=-2\frac{f_2^5 f_6^7}{f_1^4 f_4^5 f_{12}^5}\left(\dfrac{f_3^3}{f_1}-2q\dfrac{f_{12}^3}{f_4}\right)+2q\frac{f_3^6 f_{12} f_2^8}{f_1^6 f_4^7 f_6^2}\notag\\
&=-2\frac{f_2^2 f_6^{16}}{f_1^3 f_3^3 f_4^4 f_{12}^8}+2q\frac{f_3^6 f_{12} f_2^8}{f_1^6 f_4^7 f_6^2}
=-2\dfrac{f_2^8f_6^{16}}{f_1^6f_3^3f_4^7f_{12}^8}\left(\dfrac{f_1^3f_4^3}{f_2^6}-q\dfrac{f_{3}^9f_{12}^9}{f_6^{18}}\right).
\end{align}

Setting $\chi(-q):=(q;q^2)_\infty=f_1/f_2$ so that $\chi(q)=f_2^2/(f_1f_4)$ and $\chi(q)\chi(-q)=\chi(-q^2)$, we have
\begin{align}
\label{36n+25Trick2}\dfrac{f_1^3f_4^3}{f_2^6}-q\dfrac{f_{3}^9f_{12}^9}{f_6^{18}}&=\dfrac{\chi^3\left(-q\right)}{\chi^3\left(-q^2\right)}-q \dfrac{\chi^9\left(-q^3\right)}{\chi^9\left(-q^6\right)}
=\dfrac{\chi^3\left(-q\right)}{\chi^3\left(-q^2\right)}\left(1-q\dfrac{\chi^3\left(q\right)}{\chi^9\left(q^3\right)}\right).
\end{align}

Now, from \cite[p. 345, Entry 1(i)]{B1},
\begin{align}
\label{36n+5Trick3}1-q \dfrac{\chi^3\left(q\right)}{\chi^9\left(q^3\right)}=\dfrac{\chi^3\left(q\right)\psi^4\left(-q\right)}{\chi^9\left(q^3\right)\psi^4\left(-q^3\right)}=\frac{f_1 f_2^2 f_3^5 f_4 f_{12}^5}{f_6^{14}}.
\end{align}
From \eqref{36n+5Trick3}, \eqref{36n+25Trick2}, and \eqref{36n+25Trick1}, it follows that
\begin{align*}
U_{3 n+1}(\Xi_{3,2})=-2\frac{f_2^4 f_6^2 f_3^2}{f_1^2 f_4^3 f_{12}^3}.
\end{align*}
Similarly, applying the appropriate $2$- and $3$-dissections from Lemmas \ref{Lemma-2-dissection} and  \ref{Lemma-3-dissection}, it is routine to find $U_{9 n+7}(\Xi_{3,2}), $ $U_{18 n+7}(\Xi_{3,2}),$ and $U_{36 n+25}(\Xi_{3,2})$. Accordingly, we arrive at
\begin{align}
\label{The36n+5Target}U_{36 n+25}(\Xi_{3,2})&=32\frac{f_2^{10} f_3^7}{f_1^{15} f_6^2}+96 q\frac{f_2^2f_3^3 f_6^6 }{f_1^{11}}+16\frac{f_2^9 f_3^4 f_6^5}{f_1^{14} f_4^2 f_{12}^2}+64 q\frac{f_2^3f_3^6 f_4^2 f_{12}^2 }{f_1^{12} f_6}\notag\\
&\quad-48\frac{f_2^{12} f_6^{20}}{f_1^{19} f_3^5 f_{12}^8}+48q \frac{f_2^{24}f_6^8 }{f_1^{23} f_3 f_4^8}+768q^2\frac{f_4^8 f_6^8 }{f_1^{15} f_3}-768q^3\frac{f_2^{12} f_3^3 f_{12}^8 }{f_1^{19} f_6^4}.
\end{align}
Now, we proceed to show that $U_{36 n+25}(\Xi_{3,2})=0$.

From \eqref{36n+5Trick3} and \cite[pp. 345--346, Entries 1(iv) and (v)]{B1}, we have
\begin{align}
\label{36n+25Trick5}\dfrac{\psi^4\left(q\right)}{\psi^4\left(q^3\right)}+3q=4q+\dfrac{\chi^9\left(-q^3\right)}{\chi^3\left(-q\right)}=a(q)\dfrac{f_2^2f_3^3}{f_1f_6^6},
\end{align}
where $a(q)$ is the cubic theta function defined as $a(q):=\sum_{j,k=-\infty}^{\infty}q^{j^2+jk+k^2}$.

Therefore, with the help of \eqref{36n+25Trick5}, the first two terms of the right side of \eqref{The36n+5Target} become
\begin{align}
\label{FirstTwo}32\frac{f_2^{10} f_3^7}{f_1^{15} f_6^2}+96 q\frac{f_2^2f_3^3 f_6^6 }{f_1^{11}}
&=32 \dfrac{f_2^2f_3^3f_6^6}{f_1^{11}}\left(\dfrac{\psi^4(q)}{\psi^4\left(q^3\right)}+3q\right)
=32a(q)\dfrac{f_2^4f_3^6}{f_1^{12}}.
\end{align}

Now, from \cite[p. 93]{B5} and \cite[p. 232]{B1}, we recall that
\begin{align*}
a(q)&=\varphi(q)\varphi(q^3)+4q\psi(q^2)\psi(q^6),\\
4q\psi(q^2)\psi(q^6)&=\varphi(q)\varphi(q^3)-\varphi(-q)\varphi(-q^3),
\end{align*}
which can be written in the equivalent forms
\begin{align}
\label{a1inPhiPsi}a(q)&=\dfrac{f_2^5f_6^5}{f_1^2f_4^2f_3^2f_{12}^2}+4q\dfrac{f_4^2f_{12}^2}{f_2f_6},\\
\label{OddpartPhi1Phi3}4q\dfrac{f_4^2f_{12}^2}{f_2f_6}&=\dfrac{f_2^5f_6^5}{f_1^2f_4^2f_3^2f_{12}^2}-\dfrac{f_1^2f_{3}^2}{f_2f_6}.
\end{align}
Using \eqref{a1inPhiPsi}, the combined third and fourth terms of \eqref{The36n+5Target} become
\begin{align}
\label{SecondTwo}16\frac{f_2^9 f_3^4 f_6^5}{f_1^{14} f_4^2 f_{12}^2}+64 q\frac{f_2^3f_3^6 f_4^2 f_{12}^2 }{f_1^{12} f_6}&=16\dfrac{f_2^4f_3^6}{f_1^{12}}\left(\dfrac{f_2^5f_6^5}{f_1^2f_4^2f_3^2f_{12}^2}+4q\dfrac{f_4^2f_{12}^2}{f_2f_6}\right)\notag\\
&=16a(q)\dfrac{f_2^4f_3^6}{f_1^{12}}.
\end{align}

Again, with the aid of \eqref{a1inPhiPsi} and \eqref{OddpartPhi1Phi3}, the combined fifth and seventh terms of \eqref{The36n+5Target} can be simplified as
\begin{align}
\label{FifthSeventh}&-48\frac{f_2^{12} f_6^{20}}{f_1^{19} f_3^5 f_{12}^8}+768q^2\frac{f_4^8 f_6^8 }{f_1^{15} f_3}\notag\\
&=-48\dfrac{f_2^2f_4^4f_6^{10}}{f_1^{15}f_3f_{12}^4}
\left(\dfrac{f_2^5f_6^5}{f_1^2f_4^2f_3^2f_{12}^2}-4q\dfrac{f_4^2f_{12}^2}{f_2f_6}\right)\cdot\left(\dfrac{f_2^5f_6^5}{f_1^2f_4^2f_3^2f_{12}^2}+4q\dfrac{f_4^2f_{12}^2}{f_2f_6}\right)\notag\\
&=-48a(q)\dfrac{f_2f_3f_4^4f_6^{9}}{f_1^{13}f_{12}^4}.
\end{align}

Similarly, the combination of the sixth and the last terms of \eqref{The36n+5Target} becomes
\begin{align}
\label{SixthLast}48q \frac{f_2^{24}f_6^8 }{f_1^{23} f_3 f_4^8}-768q^3\frac{f_2^{12} f_3^3 f_{12}^8 }{f_1^{19} f_6^4}=48a(q)q\dfrac{f_2^{13}f_3^5f_{12}^4}{f_1^{17}f_6^3f_{4}^4}.
\end{align}

Employing \eqref{FirstTwo}, \eqref{SecondTwo}, \eqref{FifthSeventh}, and \eqref{SixthLast} in \eqref{The36n+5Target}, we find that
\begin{align*}
U_{36 n+25}(\Xi_{3,2})&=
48a(q)\dfrac{f_2^4f_3^6}{f_1^{12}}-48a(q)\dfrac{f_2f_3f_4^4f_6^{9}}{f_1^{13}f_{}12^{4}}-48qa(q)\dfrac{f_2^{13}f_3^5f_{12}^4}{f_1^{17}f_{4}^4f_6^3}.
\end{align*}
For further simplification, we rewrite the above identity as
\begin{align}
\label{36n+25ReducedTarget}U_{36 n+25}(\Xi_{3,2})&=48a(q)\dfrac{f_2^4f_3^6}{f_1^{12}}-48a(q)\dfrac{f_2^3f_3^5}{f_1^{13}f_6}\left(\psi^2\left(q^2\right)\varphi^2\left(q^3\right)-q\varphi^2\left(q\right)\psi^2\left(q^6\right)\right).
\end{align}

Now, from \cite[(3.4), (3.9)]{BBarman1}, we recall that
\begin{align}
\label{LastTrick}\varphi^2\left(q\right)+\varphi^2\left(q^3\right)&=2\varphi^2\left(-q^6\right)\dfrac{\chi\left(q\right)\psi\left(-q^3\right)}{\chi\left(-q\right)\psi\left(q^3\right)},\\
\psi^2\left(q^2\right)-q\psi^2\left(q^6\right)&=\dfrac{\varphi\left(q^3\right)\psi\left(q^3\right)}{\chi\left(q\right)\chi\left(-q^2\right)}\notag.
\end{align}
Multiplying the above two identities and then using the identities $\varphi(q)\psi(q^2)=\psi^2(q)$, $\varphi\left(-q\right)\varphi\left(q\right)=\varphi^2\left(-q^2\right)$, $\chi\left(-q\right)\chi\left(q\right)=\chi\left(-q^2\right)$, and \eqref{36n+25Trick5}, we find that
\begin{align*}
&\psi^2\left(q^2\right)\varphi^2\left(q^3\right)-q\varphi^2\left(q\right)\psi^2\left(q^6\right)\\
&=2\dfrac{\varphi^2\left(-q^6\right)\varphi\left(q^3\right)\psi\left(-q^3\right)}{\chi\left(-q\right)\chi\left(-q^2\right)}-\left(\varphi^2\left(q\right)\psi^2\left(q^2\right)-q\varphi^2\left(q^3\right)\psi^2\left(q^6\right)\right)\\
&=2\dfrac{\varphi^2\left(-q^6\right)\varphi\left(q^3\right)\psi\left(-q^3\right)}{\chi\left(-q\right)\chi\left(-q^2\right)}-\left(\psi^4\left(q\right)-q\psi^4\left(q^3\right)\right)\\
&=2\dfrac{\varphi^2\left(q^3\right)\varphi\left(-q^3\right)\psi\left(-q^3\right)}{\chi^2\left(-q\right)
\chi\left(q\right)}-\dfrac{\chi^9\left(-q^3\right)\psi^4\left(q^3\right)}{\chi^3\left(-q\right)}\\
&=\dfrac{f_3^4}{\chi^2\left(-q\right)\chi^2\left(-q^3\right)}\left(2\dfrac{\chi^3\left(q^3\right)}{\chi\left(q\right)}-\dfrac{\chi^3\left(-q^3\right)}{\chi\left(-q\right)}\right)\\
&=2\frac{ f_4 f_6^8}{f_1 f_3 f_{12}^3}-\frac{f_2^3 f_3^5}{f_1^3 f_6}.
\end{align*}
Using the above identity in \eqref{36n+25ReducedTarget}, and then applying \eqref{LastTrick} with $q$ replaced by $-q$, we have
\begin{align*}
U_{36 n+25}(\Xi_{3,2})&=
48a(q)\dfrac{f_2^4f_3^6}{f_1^{12}}\left(1+\dfrac{\varphi^2\left(-q^3\right)}{\varphi^2\left(-q\right)}-2\dfrac{f_4f_6^7}{f_1^{2}f_2f_3^2f_{12}^3}\right)\\
&=48a(q)\dfrac{f_2^4f_3^6}{f_1^{12}}\left(2\dfrac{f_4f_6^7}{f_1^{2}f_2f_3^2f_{12}^3}-2\dfrac{f_4f_6^7}{f_1^{2}f_2f_3^2f_{12}^3}\right)=0.
\end{align*}
Thus, we complete the proof of (xvi) of Theorem \ref{Theo12_1}.

For the remaining identities in Theorem \ref{Theo12_1}, we cite  the appropriate identities used to prove them in the following chart.
\begin{align*}
\begin{tabular}{c c}
 \hline
  Results in  Theorem \ref{Theo12_1}  & Used identities, like, \eqref{BBXY1} and \eqref{RobinsXY1} \T\B\\  \hline
    (i)--(v) & \cite[ Chapter 1, (1.34), (1.35)]{R1}\T\B\\
     (vi)--(x) & \cite[(30)]{BB1} and \eqref{newXY1} of Lemma \ref{NewlemmaXY}\T\B\\
 (xi)--(xiv) & \eqref{BBXY1} and \eqref{RobinsXY1} of Lemma \ref{NewlemmaXY}  \T\B\\
  (xvii)--(xxii) & \cite[(33)]{BB1} and \eqref{newXY2} of Lemma \ref{NewlemmaXY}\T\B\\\hline
\end{tabular}
\end{align*}
\qed

\section{Concluding remarks}\label{sec8}
In this paper, we present several $q$-products having matching coefficients with their reciprocals.
There might be many  more such $q$-products. Some interesting conjectures based on empirical observations are given below.
\begin{conjecture}
Suppose that
\begin{align*}
&\sum_{n=0}^{\infty}\lambda_{11}(n)q^n=\dfrac{1}{R(q)R\left(q^{14}\right)}, \quad \sum_{n=0}^{\infty}\lambda_{12}(n)q^n=\dfrac{1}{R(q)R\left(q^{2}\right)R\left(q^{4}\right)R\left(q^{8}\right)},\\
&\sum_{n=0}^{\infty}\lambda_{13}(n)q^n=\dfrac{R(q)R\left(q^{2}\right)}{R\left(q^{6}\right)R\left(q^{12}\right)}, \quad \sum_{n=0}^{\infty}\lambda_{14}(n)q^n=\dfrac{R(-q)R\left(q^{2}\right)}{R\left(-q^{6}\right)R\left(q^{12}\right)}.
\end{align*}
Then, for any $n\geq0$, we have
\begin{align*}
\lambda_{11}(10n+r)&=\lambda_{11}^{\prime}(10n+r-6),\quad r \in \{7,9,12,14\},\\
\lambda_{12}(10n+r)&=\lambda_{12}^{\prime}(10n+r-6),\quad r\in \{7,9\},\\
\lambda_{12}(20n+11)&=\lambda_{12}^{\prime}(20n+5),\\
\lambda_{13}(30n+r)&=\lambda_{13}^{\prime}(30n+r-6),\quad r\in \{10,16,20,26\},\\
\lambda_{14}(30n+r)&=\lambda_{14}^{\prime}(30n+r-6),\quad r\in \{7,17,19,29\}
\end{align*}
\end{conjecture}

The following conjectures are due to Schlosser \cite{schlosser}.

\begin{conjecture}\label{schlosser}
Suppose that $ a_{p}$ and $ b_{p}$ are  quadratic residue and non-residue modulo prime $p$, respectively, and
\begin{align*}
\sum_{n=0}^{\infty}\omega_{p}(n)q^n&=\prod_{1\le a_{p},  b_{p}\le p-1} \dfrac{\left(q^{b_{p}};q^{p}\right)_{\infty}}{\left(q^{a_{p}};q^{p}\right)_{\infty}}.
\end{align*}
Then, for any $n\geq0$, we have
\begin{align*}
\omega_{13}(13n+r)&=-\omega_{13}^{\prime}(13n+r-2), & r \in \{&3,6,7,8,9,12\},\\
\omega_{17}(17n+r)&=-\omega_{17}^{\prime}(17n+r-4),&  r \in \{&5,7,8,9,12,13,14,16\},\\
\omega_{29}(29n+r)&=-\omega_{29}^{\prime}(29n+r-6), & r \in \{&7,8,9,10,12,16,19,23,25,26,27,\\
& & &28,31,33\},\\
\omega_{53}(53n+r)&=-\omega_{53}^{\prime}(53n+r-14), & r \in \{&15,19,21,25,26,27,28,29,30,33,\\
& & &34,37,38,39,40,41,42,46,48,52,\\
& & &55,57,58,62,63,65\}.
\end{align*}
\end{conjecture}

\begin{conjecture}\label{schlosserComposite}
Suppose that
\begin{align*}
\sum_{n=0}^{\infty}\omega_{21}(n)q^n&=\dfrac{\left(q^2,q^8,q^{10},q^{11},q^{13},q^{19};q^{21}\right)_{\infty}}{\left(q,q^4,q^5,q^{16},q^{17},q^{20};q^{21}\right)_{\infty}},\\
\sum_{n=0}^{\infty}\omega_{28}(n)q^n&=\dfrac{\left(q^5,q^{11},q^{13},q^{15},q^{17},q^{23};q^{28}\right)_{\infty}}{\left(q,q^3,q^9,q^{19},q^{25},q^{27};q^{28}\right)_{\infty}},\\
\sum_{n=0}^{\infty}\omega_{77}(n)q^n&=\dfrac{\left(q^2,q^3,q^5,q^8,q^{12},q^{18},q^{20},q^{26},q^{27},
q^{29},q^{30},q^{31},q^{32};q^{77}\right)_{\infty}}{\left(q,q^4,q^6,q^{9},q^{10},q^{13},q^{15},q^{16},
q^{17},q^{19},q^{23},q^{24},q^{25};q^{77}\right)_{\infty}}\\
&\quad \times \dfrac{\left(
q^{34},q^{38},q^{39},q^{43},q^{45},q^{46},q^{47},q^{48},q^{50},q^{51},q^{57};q^{77}\right)_{\infty}}{\left(q^{36},q^{37},q^{40},q^{41},q^{52},q^{53},q^{54},q^{58},q^{60},q^{61},q^{62};q^{77}\right)_{\infty}}\\
&\quad \times \dfrac{\left(
q^{59},q^{65},q^{69},q^{72},q^{74},q^{75};q^{77}\right)_{\infty}}{\left(
q^{64},q^{67},q^{68},q^{71},q^{73},q^{76};q^{77}\right)_{\infty}}.
\end{align*}
Then, for any $n\geq0$, we have
\begin{align*}
\omega_{21}(21 n +r)&=\omega_{21}^{\prime}(21 n +r-4),& r \in \{&4,6,10,11,13,17,18,20,24\},\\
\omega_{28}(56n+r)&=\omega_{28}^{\prime}(56n+r-8),& r \in \{&14,21,38\},\\
\omega_{77}(77n+r)&=\omega_{77}^{\prime}(77n+r-24),& r \in \{&27,28,35,37,48,49,65,70,72,76,\\
& & & 79,83,90,93,98\}.
\end{align*}
\end{conjecture}

\section*{Acknowledgements}
The authors would like to thank Dr. Michael J. Schlosser from Universit\"{a}t  Wien, Austria, for his permission to include Conjectures \ref{schlosser} and \ref{schlosserComposite} in this paper. The authors are also thankful to the referee for his/her careful reading and helpful comments. The first author was partially supported by Grant no. MTR/2018/000157 of Science \& Engineering Research Board (SERB), DST, Government of India under the MATRICS scheme. The second author was partially supported by Council of Scientific \& Industrial Research (CSIR), Government of India under CSIR-JRF scheme. The authors thank both the funding agencies.

\end{document}